\def\cmntsoff{}
\documentclass[aps,superscriptaddress,12pt]{revtex4}

\usepackage[letterpaper,width=6.5in,height=9in]{geometry}

\usepackage[nolists,nomarkers,heads]{endfloat}

\usepackage{amsmath,lipsum}    
\usepackage{graphicx,xcolor}   
\usepackage{hyperref}   
\hypersetup{pdfstartview=}
\usepackage{url}
\usepackage{setspace,enumerate,amssymb,amsthm}
\usepackage{mathtools}
\usepackage[normalem]{ulem}
\usepackage{natbib}
\usepackage{stmaryrd}

\providecommand{\ignore}[1]{}

\newif\ifcmnt
\cmnttrue
\ifdefined\cmntsoff\cmntfalse\fi

\ifcmnt
    \providecommand{\aucmnt}[1]{#1}

    \providecommand{\MKcolor}{\color{brown}}
\else
    \providecommand{\aucmnt}[1]{}

    \providecommand{\MKcolor}{}
\fi
\definecolor{mygray}{rgb}{0.3,0.3,0.3}

\newcommand{\MKc}[1]{\aucmnt{{\MKcolor[{\color{mygray}MK: #1}]}}}

\newcommand{\hspc[1]}{\hspace{#1pt}}
\newcommand{\beq}{\begin{equation}}
\newcommand{\eeq}{\end{equation}}
\newcommand{\beqp}{\begin{equation*}}
\newcommand{\eeqp}{\end{equation*}}
\newcommand{\bea}{\begin{align}}
\newcommand{\eea}{\end{align}}
\newcommand{\beap}{\begin{align*}}
\newcommand{\eeap}{\end{align*}}

\newtheorem{theorem}{Theorem}
\newtheorem{lemma}[theorem]{Lemma}
\newtheorem{corollary}[theorem]{Corollary}

\newcommand{\kl}{\mathrm{KL}}

\newcommand{\PBR}{\mathrm{PBR}}
\newcommand{\CH}{\mathrm{CH}}
\newcommand{\X}{\mathrm{X}}
\newcommand{\LR}{\mathrm{LR}}
\newcommand{\lE}{\mathrm{lE}}
\newcommand{\htt}{t}
\newcommand{\htheta}{\hat{\theta}}
\newcommand{\Htheta}{\hat{\Theta}}
\newcommand{\Expt}{\mathbb{E}}
\newcommand{\Prob}{\mathbb{P}}

\newcommand{\cH}{\mathcal{H}}
\newcommand{\cB}{\mathcal{B}}

\begin{document}

\title{Performance of Test Supermartingale Confidence Intervals for the Success Probability of Bernoulli Trials}

\author{Peter Wills}
\affiliation{Department of Applied Mathematics, University of Colorado Boulder, Boulder, Colorado 80309, USA}
\author{Emanuel Knill}
\thanks{Corresponding author.}
\affiliation{National Institute of Standards and Technology, Boulder, Colorado 80305, USA}
\affiliation{Center for Theory of Quantum Matter, University of Colorado, Boulder, Colorado 80309, USA}
\author{Kevin Coakley}
\affiliation{National Institute of Standards and Technology, Boulder, Colorado 80305, USA}
\author{Yanbao Zhang}
\affiliation{Institute for Quantum Computing and Department of Physics and Astronomy, 
University of Waterloo, Waterloo, Ontario N2L 3G1, Canada}
\affiliation{NTT Basic Research Laboratories, NTT Corporation, 3-1
Morinosato-Wakamiya, Atsugi, Kanagawa 243-0198, Japan}

\begin{abstract}
  Given a composite null hypothesis $\mathcal{H}_{0}$, test
  supermartingales are non-negative supermartingales with respect to
  $\mathcal{H}_{0}$ with initial value $1$.  Large values of test
  supermartingales provide evidence against $\mathcal{H}_{0}$. As a
  result, test supermartingales are an effective tool for rejecting
  $\mathcal{H}_{0}$, particularly when the $p$-values obtained are
  very small and serve as certificates against the null
  hypothesis. Examples include the rejection of local realism as an
  explanation of Bell test experiments in the foundations of physics
  and the certification of entanglement in quantum information
  science. Test supermartingales have the advantage of being adaptable
  during an experiment and allowing for arbitrary stopping rules. By
  inversion of acceptance regions, they can also be used to determine
  confidence sets.  We use an example to compare the performance of
  test supermartingales for computing $p$-values and confidence
  intervals to Chernoff-Hoeffding bounds and the``exact''
  $p$-value. The example is the problem of inferring the probability
  of success in a sequence of Bernoulli trials.  There is a cost in
  using a technique that has no restriction on stopping rules, and for
  a particular test supermartingale, our study quantifies this cost.
\end{abstract}

\ignore{
Given a composite null hypothesis H, test supermartingales are non-negative supermartingales with respect to H with initial value 1.  Large values of test supermartingales provide evidence against H. As a result, test supermartingales are an effective tool for rejecting H, particularly when the p-values obtained are very small and serve as certificates against the null hypothesis. Examples include the rejection of local realism as an explanation of Bell test experiments in the foundations of physics and the certification of entanglement in quantum information science. Test supermartingales have the advantage of being adaptable during an experiment and allowing for arbitrary stopping rules. By inversion of acceptance regions, they can also be used to determine confidence sets.  We use an example to compare the performance of test supermartingales for computing p-values and confidence intervals to Chernoff-Hoeffding bounds and the``exact'' p-value. The example is the problem of inferring the probability of success in a sequence of Bernoulli trials.  There is a cost in using a technique that has no restriction on stopping rules, and for a particular test supermartingale, our study quantifies this cost.
}

\maketitle
\tableofcontents

\section{Introduction}
\label{s:intro}

Experiments in physics require very high confidence to justify claims
of discovery or to unambiguously exclude alternative
explanations~\cite{dorigo2015extraordinary}. Particularly striking
examples in the foundations of physics are experiments to demonstrate
that theories based on local hidden variables, called local realist
(LR) theories, cannot explain the statistics observed in quantum
experiments called Bell tests. See Ref.~\cite{genovese2005research}
for a review and Refs.~\cite{hensen:2015, shalm:2015,
  giustina:2015,rosenfeld:2016} for the most definitive experiments to
date. Successful Bell tests imply the presence of some randomness in
the observed statistics. As a result, one of the most notable
applications of Bell tests is to randomness
generation~\cite{acin:qc2016a}. In this application, it is necessary
to certify the randomness generated, and these certificates are
equivalent to extremely small significance levels in an appropriately
formulated hypothesis test. In general, such extreme significance
levels are frequently required in protocols for communication or
computation to ensure performance.

Bell tests consist of a sequence of ``trials'', each of which gives a
result $M_{i}$.  LR models restrict the statistics of the $M_{i}$ and
therefore constitute a composite null hypothesis to be
rejected. Traditionally, data has been analyzed by estimating the
value of a Bell function $\hat B$ and its standard error $\hat\sigma$
from the collective result statistics
(see~\cite{PRA84,larsson:qc2014a}). Under the null hypothesis, $\hat
B$ is expected to be negative, so a large value of $\hat B$ compared
to $\hat\sigma$ is considered strong evidence against the null
hypothesis. This method suffers from several problems, including the
failure of the Gaussian approximation in the extreme tails and the
fact that the trials are observably not independent and identically
distributed (i.i.d.)~\cite{PRA84}.

In Ref.~\cite{PRA84} a method was introduced that can give rigorous
$p$-value bounds against LR. These $p$-value bounds are memory-robust,
that is, without any assumptions on dependence of trial statistics on
earlier trials.  The method can be seen as an application of test
supermartingales as defined in Ref.~\cite{Shafer}. Test
supermartingales were first considered, and many of their basic
properties were proved, by Ville~\cite{ville:1939a} in the same work
that introduced the notion of martingales. The method involves
constructing a non-negative stochastic process $V_{i}$ determined by
$(M_{j})_{j\leq i}$ such that the initial value is $V_0=1$ and, under
LR models, the expectations conditional on all past events are
non-increasing.  As explained further below, the final value of
$V_{i}$ in a sequence of $n$ trials has expectation bounded by $1$, so
its inverse $p=1/V_{n}$ is a $p$-value bound according to Markov's
inequality. A large observed value of such a test supermartingale thus
provides evidence against LR models.  Refs.~\cite{PRA84,PRA88} give
methods to construct $V_{i}$ that achieve asymptotically optimal gain
rate $\Expt(-\log(p)/n)$ for i.i.d.~trials, where $\Expt(\ldots)$ is
the expectation functional. This is typically an improvement over
other valid memory-robust Bell tests. Additional benefits are that
$V_{i}$ can be constructed adaptively based on the observed
statistics, and the $p$-value bounds remain valid even if the
experiment is stopped based on the current value of $V_{i}$.  These
techniques were successfully applied to experimental data from a Bell
test with photons where other methods
fail~\cite{Brad2015}.

Although the terminology is apparently relatively recent, test
supermartingales have traditionally played a major theoretical
role. Carefully constructed test supermartingales contribute to the
asymptotic analysis of distributions and proofs of large deviation
bounds.  They can be constructed for any convex-closed null hypothesis
viewed as a set of distributions, so they can be used for memory-
and stopping-robust adaptive hypothesis tests in some generality. The
application to Bell tests shows that at least in a regime where high
significance results are required, test supermartingales can perform
well or better than other methods. Here we compare the performance of
test supermartingales directly to (1) the standard large deviation
bounds based on the Chernoff-Hoeffding
inequality~\cite{Chernoff,Hoeffding}, and (2) ``exact'' $p$-value
calculations. Our comparison is for a case where all calculations can
be performed efficiently, namely for testing the success probability
in Bernoulli trials.  The three $p$-value bounds thus obtained have
asymptotically optimal gain rates.  Not surprisingly, for any given
experiment, test supermartingales yield systematically worse $p$-value
bounds, but the difference is much smaller than the
experiment-to-experiment variation. This effect can be viewed as the
cost of robustness against arbitrary stopping rules.  For ease of
calculation, we do not use an optimal test supermartingale
construction, but we expect similar results no matter which test
supermartingale is used.

Any hypothesis test parametrized by $\phi$ can be used to construct
confidence regions for $\phi$ by acceptance region inversion (see
Ref.~\cite{Shao}, Sect. 7.1.2).  Motivated by this observation, we
consider the use of test supermartingales for determining confidence
regions. We expect that they perform well in the high-confidence
regime, with an increase in region size associated with robustness
against stopping rules. We therefore compared the methods mentioned
above for determining confidence intervals for the success probability
in Bernoulli trials. After normalizing the difference between
the interval endpoints and the
success probability by the standard deviation, which is
$O(1/\sqrt{n})$, we find that while large deviation bounds and exact
regions differ by a constant at fixed confidence levels, the test
supermartingale's normalized endpoint deviation is
$\Omega(\sqrt{\log(n)})$ instead of $O(1)$. This effect was noted in
Ref.~\cite{Shafer} and partially reflects a suboptimal choice of
supermartingale. To maintain robustness against stopping rule, one
expects $\Omega(\sqrt{\log\log(n)})$ according to the law of the
iterated logarithm.  However, we note that if the
number of trials $n$ is fixed in advance, the normalized endpoint
deviation can be reduced to $O(1)$ with an adaptive test
supermartingale. So although the increase in confidence region
necessitated by stopping rule robustness is not so large for
reasonably sized $n$, when $n$ is known ahead of time it can, in
principle, be avoided without losing the ability to adapt the test
supermartingale on the fly during the experiment in
non-i.i.d.~situations.

The remainder of the paper is structured as follows. We establish the
notation to be used and define the basic concepts in
Sect.~\ref{s:basic}. Here we also explain how adaptivity can help
  reject hypotheses for stochastic processes.  We introduce the three
methods to be applied to Bernoulli trials in Sect.~\ref{s:bht}. Here
we also establish the basic monotonicity properties and relationships
of the three $p$-value bounds obtained.  In Sect.~\ref{s:pvalcomp} we
determine the behavior of the $p$-value bounds in detail, including
their asymptotic behavior.  In Sect.~\ref{s:cicomp} we introduce the
confidence intervals obtained by acceptance region inversion. We focus
on one-sided intervals determined by lower bounds but note that the
results apply to two-sided intervals. The observations in
Sects.~\ref{s:pvalcomp} and~\ref{s:cicomp} are based on theorems whose
proofs can be found in the Appendix. While many of the observations in
these sections can ignore asymptotically small terms, the results in
the Appendix uncompromisingly determine interval bounds for all
relevant expressions, with explicit constants. Concluding remarks can
be found in Sect.~\ref{s:conclusion}.

\section{Basic Concepts}
\label{s:basic}

We use the usual conventions for random variables (RVs) and their
values.  RVs are denoted by capital letters such as $X,Y,\ldots$ and
their values by the corresponding lower case letters $x,y,\ldots$.
All our RVs are finite valued.  Probabilities and expectations are
denoted by $\Prob(\ldots)$ and $\Expt(\ldots)$, respectively. For a
formula $\phi$, the expression $\{\phi\}$ refers to the event where
the formula is true.  The notation $\mu(X)$ refers to the distribution
of $X$ induced on its space of values.  We use the usual conventions
for conditional probabilities and expectations.  Also, $\mu(X|\phi)$
denotes the probability distribution induced by $X$ conditional on the
event $\{\phi\}$.  

We consider stochastic sequences of RVs such as
$\mathbf{X}=(X_{i})_{i=1}^{n}$ and $\mathbf{X}_{\leq
  k}=(X_{i})_{i=1}^{k}$.  We think of the $X_{i}$ as the outcomes from
a sequence of trials. For our study, we consider
  $\mathbf{B}=(B_{i})_{i=1}^{n}$, where the $B_{i}$ are
  $\{0,1\}$-valued RVs. The standard $\{0,1\}$-valued RV with
  parameter $\theta$ is the Bernoulli RV $B$ satisfying
  $\Expt(B)=\theta$. The parameter $\theta$ is also referred to as the
  success probability.  We denote the distribution of $B$ by
$\nu_{\theta}$.  We define $S_{k}=\sum_{i=1}^{k}B_{i}$ and
$\Htheta_{k}=S_{k}/k$.  We extend the RV conventions to the Greek
letter $\Htheta_{k}$.  That is,
$\htheta_{k}=s_{k}/k=\sum_{i=1}^{k}b_{i}/k$ is the value of the RV
$\Htheta_{k}$ determined by the values $b_{i}$ of $B_{i}$.  We may
omit subscripts on statistics such as $S_{n}$ and $\Htheta_{n}$ when
they are based on the full set of $n$ samples. Some expressions
involving $\Htheta_{n}$ require that $n\Htheta_{n}$ is an integer,
which is assured by the definition.

A null hypothesis for $X$ is equivalent to a set $\cH_{0}$ of
distributions of $X$, which we refer to as the ``null''.  For our
study of Bernoulli RVs, we consider the nulls
\begin{equation}
  \cB_{\varphi}=\{\nu_{\theta}|\theta\leq\varphi\}
\end{equation}
parametrized by $0\leq \varphi\leq 1$.  the set of distributions of
Bernoulli RVs with $\Prob(B=1)\leq\varphi$.  One can test the null
hypothesis determined by a null by means of special statistics called
$p$-value bounds. A statistic $P=P(X)\geq 0$ is a $p$-value bound for
$\cH_{0}$ if for all $\mu\in \cH_{0}$ and $p\geq 0$,
$\Prob_{\mu}(P\leq p)\leq p$.  Here, the subscript $\mu$ on
$\Prob_{\mu}(\ldots)$ indicates the distribution with respect to which
the probabilities are to be calculated.  We usually just write
``$p$-value'' instead of ``$p$-value bound'', even when the bounds are
not achieved by a member of $\cH_{0}$. Small $p$-values are strong
evidence against the null. Since we are interested in very small
$p$-values, we preferentially use their negative logarithm $-\log(P)$
and call this the $\log(p)$-value. In this work, logarithms are base
$e$ by default.

A general method for constructing $p$-values is to start with an
arbitrary real-valued RV $Q$ jointly distributed with $X$.  Usually
$Q$ is a function of $X$. Define the worst-case tail probability of
$Q$ as $P(q)=\sup_{\mu\in\cH_{0}}\Prob_{\mu}(Q\geq q)$. Then
$P(Q)$ is a $p$-value for $\cH_{0}$.  The argument is standard.
Define $F_{\mu}(q)=\Prob_{\mu}(Q\geq q)$.  The function
$F_{\mu}$ is non-increasing. We need to show that for all
$\mu\in\cH_{0}$, $\Prob_{\mu}(P(Q)\leq p)\leq p$.  Since
$F_{\mu}(q)\leq P(q)$, we have $\Prob_{\mu}(P(Q)\leq p) \leq
\Prob_{\mu}(F_{\mu}(Q)\leq p)$.  The set
$\{q:F_{\mu}(q)\leq p\}$ is either of the form $[q_{\min},\infty)$
or $(q_{\min},\infty)$ for some $q_{\min}$. In the first case,
$\Prob_{\mu}(F_{\mu}(Q)\leq p) = \Prob_{\mu}(Q\geq
q_{\min}) = F_{\mu}(q_{\min})\leq p$.  In the second,
$\Prob_{\mu}(F_{\mu}(Q)\leq p) =
\Prob_{\mu}\left(\bigcup_{n}\{q:q\geq q_{\min}+1/n\}\right) =
\lim_{n}\Prob_{\mu}(\{q:q\geq
q_{\min}+1/n\})=\lim_{n}\Prob_{\mu}(F_{\mu}(Q)\leq
q_{\min}+1/n) \leq p$, with $\sigma$-additivity applied to the
countable monotone union.

When referring to $\cH_{0}$ as a null for $\mathbf{X}$, we mean that
$\cH_{0}$ consists of the distributions where the $X_{i}$ are
i.i.d., with $X_{i}$ distributed according to $\mu$ for some
fixed $\mu$ independent of $i$.  To go beyond i.i.d., we extend
$\cH_{0}$ to the set of distributions of $\mathbf{X}$ that have the
property that for all $\mathbf{x}_{\leq i-1}$,
$\mu(X_{i}|\mathbf{X}_{\leq i-1}=\mathbf{x}_{\leq i-1}) =\mu_{i}$
for some $\mu_{i}\in \cH_{0}$, where $\mu_{i}$ depends on $i$ and
$\mathbf{x}_{\leq i-1}$. We denote the extended null by $\overline{\cH_{0}}$.
In particular,
\begin{equation}
  \overline{\cB_{\varphi}}=\{\mu:\textrm{for all $i$ and $\mathbf{b}_{\leq i-1}$, $\mu(B_{i}|\mathbf{B}_{\leq i-1}=\mathbf{b}_{\leq i-1})
    =\nu_{\theta}$ for some $\theta\leq\varphi$}\}.
\end{equation}

The LR models mentioned in the introduction constitute a particular
null $\cH_{\LR}$ for sequences of trials called Bell tests. In
Ref.~\cite{PRA84}, a technique called the probability-based ratio
(PBR) method was introduced to construct $p$-values $P_{n}$ that
achieve asymptotically optimal gain rates defined as
$\Expt(\log(1/P_{n}))/n$. The method is best understood as a way of
constructing a test supermartingale for $\cH_{\LR}$.  A test
supermartingale of $\mathbf{X}$ for $\cH_{0}$ is a stochastic sequence
$\mathbf{T}=(T_{i})_{i=0}^{n}$ where $T_{i}$ is a function of
$\mathbf{X}_{\leq i}$, $T_{0}=1$, $T_{i}\geq 0$ and for all
$\mu\in\cH_{0}$, $\Expt_{\mu}(T_{i+1}|\mathbf{X}_{\leq i})\leq
T_{i}$. In this work, to avoid unwanted boundary cases, we further
require $T_{i}$ to be positive. The definition of test supermartingale
used here is not the most general one because we consider only
discrete time and avoid the customary increasing sequence of
$\sigma$-algebra by making it dependent on an explicit stochastic
sequence $\mathbf{X}$.  Every test supermartingale defines a $p$-value
by $P_{n}=1/T_{n}$. This follows from $\Expt(T_{n})\leq T_{0}=1$ (one
of the characteristic properties of supermartingales) and Markov's
inequality for non-negative statistics, according to which
$\Prob(T_{n}\geq \kappa)\leq \Expt(T_{n})/\kappa\leq 1/\kappa$. From
martingale theory, the stopped process $T_{\tau}$ for any stopping
rule $\tau$ with respect to $\mathbf{X}$ also defines a $p$-value by
$P=1/T_{\tau}$.  Further, $P^{*}_{n}=1/\max_{i=1}^{n}T_{i}$ also
defines a $p$-value. See Ref.~\cite{Shafer} for a discussion and
examples.

A test supermartingale $\mathbf{T}$ can be viewed as the running
product of the $F_{i}=T_{i}/T_{i-1}$, which we call the test factors
of $\mathbf{T}$.  The defining properties of $\mathbf{T}$ are
equivalent to having $F_{i}>0$ and $\Expt(F_{i}|\mathbf{X}_{\leq
  i-1})\leq 1$ for all distributions in the null, for all $i$. The PBR
method adaptively constructs $F_{i}$ as a function of the next trial
outcome $X_{i}$ from the earlier trial outcomes $\mathbf{X}_{\leq
  i-1}$.  The method is designed for testing $\overline{\cH_{0}}$ for
a closed convex null $\cH_{0}$, where asymptotically optimally
gain rates are achieved when the trials are i.i.d.~with a trial
distribution $\nu$ not in $\cH_{0}$. If $\nu$ were known, the optimal
test factor would be given by $x\mapsto \nu(x)/\mu(x)$, where
$\mu\in\cH_{0}$ is the distribution in $\cH_{0}$ closest to $\nu$ in
Kullback-Leibler (KL) divergence
$\kl(\nu|\mu)=\sum_{x}\nu(x)\log(\nu(x)/\mu(x))$~\cite{KL1951}. Since
$\nu$ is not known, the PBR method obtains an empirical estimate
$\hat\nu$ of $\nu$ from $\mathbf{x}_{\leq i-1}$ and other information
available before the $i$'th trial. It then determines the KL-closest
$\mu\in\cH_{0}$ to $\hat\nu$. The test factor $F_{i}$ is then given by
$F_{i}(x)=\hat\nu(x)/\mu(x)$. The test factors satisfy
$\Expt_{\mu'}(F_{i})\leq 1$ for all $\mu'\in\cH_{0}$, see
Ref.~\cite{PRA84} for a proof and applications to the problem of
testing LR.

The ability to choose test factors adaptively helps reject extended
nulls when the distributions vary as the experiment progresses, both
when the distributions are still independent (so only the parameters
vary) and when the parameters depend on past outcomes. Suppose that
the distributions are sufficiently stable so that the empirical
frequencies over the past $k$ trials are statistically close to the
next trial's probability distribution. Then we can adaptively compute
the test factor to be used for the next trial from the past $k$
trials' empirical frequencies, for example by following the strategy
outlined in the previous paragraph. The procedure now has an
opportunity to reject an extended null provided only that there is a
sufficiently long period where the original null does not hold. For
example, consider the extended null $\overline{\cB_{\varphi}}$. The
true success probabilities $\theta_{i}$ at the $i$'th trial may vary,
maybe as a result of changes in experimental parameters that need to
be calibrated. Suppose that the goal is to calibrate for
$\theta_{i}>\varphi$. If we use adaptive test factors and find at some
point that we cannot reject $\overline{\cB_{\varphi}}$ according to
the running product of the test factors, we can recalibrate during the
experiment. If the the recalibration succeeds at pushing $\theta_{i}$
above $\varphi$ for the remaining trials, we may still reject the
extended null by the end of the experiment.  In many cases, the
analysis is performed after the experiment, or it may not be possible
to stop the experiment for recalibration. For this situation, if the
frequencies for a run of $k$ trials clearly show that
$\theta_{i}<\varphi$, the adaptive test factors chosen would tend to
be trivial (equal to $1$), in which case the next trials do not
contribute to the final test factor product.  This is in contrast to a
hypothesis test based on the final sum of the outcomes for which all
trials contribute equally.

Let $\varphi$ be a parameter of distributions of $X$. Here, $\varphi$
need not determine the distributions.  There is a close relationship
between methods for determining confidence sets for $\varphi$ and
hypothesis tests. Let $\cH_{\varphi}$ be a null such that for all
distributions $\mu$ with parameter $\varphi$,
$\mu\in\cH_{\varphi}$. Given a family of hypothesis tests with
$p$-values $P_{\varphi}$ for $\cH_{\varphi}$, we can construct
confidence sets for $\varphi$ by inverting the acceptance region of
$P_{\varphi}$, see Ref.~\cite{Shao}, Sect. 7.1.2. According to this
construction, the confidence set $C_{a}$ at level $a$ is given by
$\{\varphi|P_{\varphi}(X)\geq a\}$ and is a random quantity.  The
defining property of a level $a$ confidence set is that its coverage
probability satisfies $\Prob_{\mu}(\varphi\in C_{a})\geq 1-a$ for all
distributions $\mu\in\cH_{\varphi}$.  When we use this construction
for sequences $\mathbf{B}$ of i.i.d.~Bernoulli RVs with the null
$\cB_{\varphi}$, we obtain one-sided confidence intervals of the form
$[\varphi_{0},1]$ for $\theta=\Expt(B_{i})$. When the confidence set
is a one-sided interval of this type, we refer to $\varphi_{0}$ as the
confidence lower bound or endpoint.  If $\mathbf{B}$ has a
distribution $\mu$ that is not necessarily i.i.d., we can define
$\Theta_{\max}=\max_{i\leq n}\Expt_{\mu}(B_{i}|\mathbf{B}_{\leq
  i-1})$.  If we use acceptance region inversion with the extended
null $\overline{\cB_{\varphi}}$, we obtain a confidence region for
$\Theta_{\max}$. Note that $\Theta_{\max}$ is an RV whose value is
covered by the confidence set with probability at least $1-a$.  The
confidence set need not be an interval in general, but including
everything between its infimum and its supremum increases the coverage
probability, so the set can be converted into an interval if desired.

While our focus is on one-sided confidence intervals, our observations
immediately apply to two-sided intervals ones with a standard method
of obtaining a two-sided confidence interval from two one-sided
intervals. For our example, we can obtain confidence upper bounds at
level $a$ by symmetry, for example by relabeling the Bernoulli
outcomes $0\mapsto 1$ and $1\mapsto 0$.  To obtain a two-sided
interval at level $a$, we compute lower and upper bounds at level
$a/2$. The two-sided interval is the interval between the bounds.  The
coverage probability of the two-sided interval is valid according to
the union bound applied to maximum non-coverage probabilities of the
two one-sided intervals.

\section{Bernoulli Hypothesis Tests}
\label{s:bht}

We compare three hypothesis tests for the nulls $\cB_{\varphi}$ or the
extended nulls $\overline{\cB_{\varphi}}$: The ``exact'' test with
$p$-value $P_{\X}$, the Chernoff-Hoeffding test with $p$-value
$P_{\CH}$ and a PBR test with $p$-value $P_{\PBR}$.  In discussing
properties of these tests with respect to the hypothesis parameter
$\varphi$, the true success probability $\theta$ and the empirical
success probability $\Htheta$, we generally assume that these
parameters are in the interior of their range.  In particular,
$0<\varphi<1$, $0<\theta<1$, and $0<\Htheta<1$.  When discussing
purely functional properties with respect to values $\htheta$ of
$\Htheta$, we use the variable $\htt$ instead of $\htheta$. By default
$n\htt$ is a positive integer.

The $p$-value for the exact test is obtained from the tail for
i.i.d.~Bernoulli RVs:
\begin{equation}
 P_{\X,n}(\Htheta|\varphi) = \sum_{k\geq \Htheta n}\binom{n}{k}\varphi^{k}(1-\varphi)^{n-k},
\end{equation}
where $\Htheta=S_{n}/n=\sum_{i=1}^{n}B_{i}/n$ as defined in
Sect.~\ref{s:basic}. Note that unlike the other $p$-values we
consider, $P_{\X,n}$ is not just a $p$-value bound. It is achieved by
a member of the null.  The quantity $P_{\X,n}(\htt|\varphi)$ is decreasing as a
function of $\htt$, given $0<\varphi<1$. It is smooth and
monotonically increasing as a function of $\varphi$, given $t>0$.  To
see this, compute
\begin{align}
  \frac{d}{d\varphi}P_{\X,n}(\htt|\varphi) &= \sum_{i=n\htt}^n \varphi^i(1-\varphi)^{n-i}\binom{n}{i}\left(\frac{i}{\varphi}-\frac{n-i}{1-\varphi}\right)\nonumber\\
  &= n\sum_{i=n\htt}^n \varphi^{i-1}(1-\varphi)^{n-i}\binom{n-1}{i-1}-n\sum_{i=n\htt}^{n-1} \varphi^{i}(1-\varphi)^{n-i-1}\binom{n-1}{i}\nonumber\\
  &= n \left(\sum_{i=n\htt-1}^{n-1} \varphi^{i}(1-\varphi)^{n-1-i}\binom{n-1}{i}-\sum_{i=nt}^{n-1} \varphi^{i}(1-\varphi)^{n-1-i}\binom{n-1}{i}\right)\nonumber\\
  &= n\,\varphi^{n\htt-1}(1-\varphi)^{n(1-\htt)}\binom{n-1}{n\htt-1}.
\end{align}
This is positive for $\varphi\in(0,1)$.  The probability that
$S_{n}\geq \htt n$, given that all $B_{i}$ are distributed as
$\nu_{\theta}$ with $\theta\leq\varphi$, is bounded by
$P_{\X,n}(\htt|\theta)\leq P_{\X,n}(\htt|\varphi)$.  That $P_{\X}$ is
a $p$-value for the case where the null is restricted to
i.i.d.~distributions now follows from the standard construction of
$p$-values from worst-case (over the null) tails of
statistics (here $S_{n}$) as explained in the previous section.  That
$P_{\X}$ is a $p$-value for the extended null
$\overline{\cB_{\varphi}}$ follows from the observations that the tail
probabilities of $S_{n}$ are linear functions of the distribution
parameters $\theta_1, \theta_2, ...,\theta_n$ where $\theta_i\leq
\varphi, i=1,2,...,n$, the extremal distributions in
$\overline{\cB_{\varphi}}$ have $B_{i}$ independent with
$\Prob(B_{i}=1)=\theta_{i}\leq\varphi$, and the tail probabilities of
$S_{n}$ are monotonically increasing in $\Prob(B_{i}=1)$ for each $i$
separately. See also Ref.~\cite{bierhorst:qc2015a}, App. C.

Define $\Htheta_{\max}=\max(\Htheta,\varphi)$.
The $p$-value for the Chernoff-Hoeffding test is the optimal
Chernoff-Hoeffding bound~\cite{Chernoff,Hoeffding} for a binary 
random variable given by
\begin{align}
  P_{\CH,n}(\Htheta|\varphi) &= \left(\frac{\varphi}{\Theta_{\max}}\right)^{n\Theta_{\max}}\left(\frac{1-\varphi}{1-\Theta_{\max}}\right)^{n(1-\Theta_{\max})}\notag\\
  &= \left\{ \begin{array}{ll}
      \left(\frac{\varphi}{\Htheta}\right)^{n\Htheta}\left(\frac{1-\varphi}{1-\Htheta}\right)^{n(1-\Htheta)} & \textrm{if $\Htheta\geq\varphi$,}\\
      1 &\textrm{otherwise.}
    \end{array}
  \right.
\end{align}
This is a $p$-value for our setting because
$P_{\CH,n}(\htt|\varphi)\geq P_{\X,n}(\htt|\varphi)$, see
Ref.~\cite{Hoeffding}. For $\varphi\leq \htt$, we have
$-\log(P_{\CH,n}(\htt|\varphi))=n\kl(\nu_{\htt}|\nu_{\varphi})$.  We
abbreviate $\kl(\nu_{\htt}|\nu_{\varphi})$ by $\kl(\htt|\varphi)$. 
For $\varphi\leq\htt<1$, $P_{\CH,n}(\htt|\varphi)$ is
monotonically increasing in $\varphi$, and decreasing in $\htt$.  For
$0\leq\htt\leq\varphi$, it is constant.

\MKc{Permanent comment to verify the monotonicity stated: From
  $\kl(\htt|\varphi)=\htt\log\htt/\varphi)+(1-\htt)\log((1-\htt)/(1-\varphi))$,
  for $\varphi<\htt$
  \begin{align*}
   \frac{d}{d\varphi}\kl(\htt|\varphi) &=
   -\frac{\htt}{\varphi}+\frac{1-\htt}{1-\varphi}\\
  &= \frac{-\htt(1-\varphi)+(1-\htt)\varphi}{\varphi(1-\varphi)}\\
  &= \frac{\varphi-\htt}{\varphi(1-\varphi)}\\
  &<0
   \\
   \frac{d}{d\htt}\kl(\htt|\varphi)&=
   \log\left(\frac{\htt}{\varphi}\right)-\log\left(\frac{1-\htt}{1-\varphi}\right)\\
   &=\log\left(\frac{\htt}{1-\htt}\middle/\frac{\varphi}{1-\varphi}\right)\\
   &>\log(1)=0.
  \end{align*}
}

The $p$-value for the PBR test that we use for comparison is
constructed from a $p$-value for the point null $\{\nu_{\varphi}\}$
defined as
\begin{align}
  P^{0}_{\PBR,n}(\Htheta|\varphi) &= \varphi^{n\Htheta}(1-\varphi)^{n(1-\Htheta)}(n+1)\binom{n}{n\Htheta}.
\end{align}
The PBR test's $p$-value for $\overline{\cB_{\varphi}}$ is
\begin{align}
  P_{\PBR,n}(\Htheta|\varphi) 
  &= \max_{0\leq\varphi'\leq\varphi} P^{0}_{\PBR,n}(\Htheta|\varphi').
  \label{e:pbrmaindef}
\end{align}
That $P_{\PBR}$ is a $p$-value for $\overline{\cB_{\varphi}}$ is
shown below.  As a function of $\varphi$,
$P^{0}_{\PBR,n}(\htt|\varphi)$ has an isolated maximum at
$\varphi=\htt$. This can be seen by differentiating
$\log\left(\varphi^{\htt}(1-\varphi)^{1-\htt}\right)=\htt\log(\varphi)+(1-\htt)\log(1-\varphi)$.
Thus in Eq.~\ref{e:pbrmaindef} when
$\varphi\geq\Htheta$, the maximum is achieved by $\varphi'=\Htheta$.
We can therefore write Eq.~\ref{e:pbrmaindef} as
\begin{align}
  P_{\PBR,n}(\Htheta|\varphi) &= 
  \left\{ \begin{array}{ll}
      P^{0}_{\PBR,n}(\Htheta|\varphi) & \textrm{if $\Htheta\geq\varphi$,}\\
      P^{0}_{\PBR,n}(\Htheta|\Htheta) &\textrm{otherwise.}
    \end{array}
  \right.
\end{align}
By definition, $P_{\PBR,n}(\htt|\varphi)$ is non-decreasing in
$\varphi$ and strictly increasing for $\varphi\leq\htt$. As a
function of $\htt$, it is strictly decreasing for $\htt\geq\varphi$
(integer-valued $n\htt$). To see this, consider $k=n\htt\geq n\varphi$
and compute the ratio of successive values as follows:
\begin{align}
  P^{0}_{\PBR,n}((k+1)/n|\varphi)/P^{0}_{\PBR,n}(k/n|\varphi)
  &= \frac{\varphi}{1-\varphi}\frac{n-k}{k+1}\notag\\
  &= \frac{\varphi}{1-\varphi}\frac{1-\htt}{\htt+1/n}\notag\\
  &\leq \frac{\varphi}{1-\varphi}\frac{1-\htt}{\htt}\notag\\
  &\leq 1.\label{e:decreasingbinom}
\end{align}

The expression for $P^{0}_{\PBR,n}$ is the final value of a test
supermartingale obtained by constructing test factors $F_{k+1}$ from
$S_{k}$.  Define
\begin{equation}
  \widetilde\Theta_k = \frac{1}{k+2}\left(S_{k} +
    1\right).
\end{equation}
Thus, $\widetilde\Theta_{k}$ would be an empirical estimate of
$\theta$ if there were two initial trials $B_{-1}$ and $B_{0}$ with
values $0$ and $1$, respectively. The test factors are given by
\begin{equation}
  F_{k+1}(B_{k+1}) = \left(\frac{\widetilde{\Theta}_{k}}{\varphi}\right)^{B_{k+1}}
  \left(\frac{1-\widetilde{\Theta}_{k}}{1-\varphi}\right)^{1-B_{k+1}}.
  \label{eq:pbr_testfacts}
\end{equation}
One can verify that $\Expt_{\nu_{\theta}}(F_{k+1})= 1$ for
$\theta=\varphi$.  More generally, set $\delta=\theta-\varphi$ and
compute
\begin{align}
  \Expt_{\nu_{\theta}}(F_{k+1}|\widetilde{\Theta}_{k}=t) &= \theta\frac{t}{\varphi}
  + (1-\theta)\frac{1-t}{1-\varphi}\notag\\
  &= 1 + \delta\left(\frac{t}{\varphi}-\frac{1-t}{1-\varphi}\right)\notag\\
  &= 1 + \delta\frac{t-\varphi}{\varphi(1-\varphi)}\label{eq:fk_isonesided}.
\end{align}
As designed, $T_{n}=\prod_{k=1}^{n}F_{k}$ is a test supermartingale
for the point null $\{\nu_{\varphi}\}$. Thm.~\ref{t:pbrexpression} in
App.~\ref{a:pbrexpression}, establishes that
$T_{n}=1/P^{0}_{\PBR,n}(\Htheta|\varphi)$.  The definition of
$P_{\PBR,n}(\Htheta|\varphi)$ as a maximum of $p$-values for
$\nu_{\varphi'}$ with $\varphi'\leq\varphi$ in Eq.~\ref{e:pbrmaindef}
ensures that $P_{\PBR,n}(\Htheta|\varphi)$ is a $p$-value for
$\cB_{\varphi}$.

To show that $P_{\PBR}$ is a $p$-value for $\overline{\cB_{\varphi}}$,
we establish that for all $\htt$ (integer-valued $n\htt$),
$P_{\PBR,n}(\htt|\varphi)\geq P_{\CH,n}(\htt|\varphi)$.  By direct
calculation for both $\varphi\leq\htt$ and $\htt\leq\varphi$, we
have
\begin{equation}
P_{\PBR,n}(\htt|\varphi)/P_{\CH,n}(\htt|\varphi) =  \htt^{n\htt}(1-\htt)^{n(1-\htt)}(n+1)\binom{n}{n\htt}.
\end{equation}
The expression $\htt^{k}(1-\htt)^{k}\binom{n}{k}$ is maximized at
$k=n\htt$ as can be seen by considering ratios for successive values
of $k$ and the calculation in Eq.~\ref{e:decreasingbinom}, now
applied also for $k<n\htt$. Therefore,
\begin{equation}
  \htt^{n\htt}(1-\htt)^{n(1-\htt)}(n+1)\binom{n}{n\htt} =
  \sum_{k=0}^{n}\htt^{n\htt}(1-\htt)^{n(1-\htt)}\binom{n}{n\htt}\geq
  \sum_{k=0}^{n}\htt^{k}(1-\htt)^{k}\binom{n}{k} = 1.
\end{equation}

A better choice for test factors to construct a test supermartingale
to test $\overline{\cB_{\varphi}}$ would be
\begin{equation}
  T'_{k+1}=\left\{\begin{array}{ll}
      T_{k+1}& \textrm{if $\widetilde{\Theta}_{k}\geq\varphi$,}\\
      1 & \textrm{otherwise.}
    \end{array}\right.
\end{equation}
This choice ensures that $\Expt_{\nu_{\theta}}(F_{k+1}|\mathbf{B}_{\le
  k})\leq 1$ for all $\theta\leq\varphi$ but the final value of the test
supermartingale obtained by multiplying these test factors is not
determined by $S_{n}$, which would complicate our study.

We summarize the observations about the three tests in the following theorem.

\begin{theorem}\label{t:pproperties}
  We have
  \begin{equation}
    P_{\X}\leq P_{\CH}\leq P_{\PBR}.
  \end{equation}
  The three tests satisfy the following monotonicity properties for $0<\varphi<1$
  and $0<\htt<1$ with $n\htt$ integer-valued:
  \begin{itemize}
  \item[] 
    $P_{\X}(\htt|\varphi)$ is strictly increasing in $\varphi$
    and strictly decreasing as a function of $\htt$.
  \item[] 
    $P_{\CH}(\htt|\varphi)$ is strictly increasing in $\varphi$ for
    $\varphi\leq\htt$, constant in $\varphi$ for $\varphi\geq\htt$,
    strictly decreasing in $\htt$ for $\htt\geq\varphi$ and
    constant in $\htt$ for $\htt\leq\varphi$.
  \item[] $P_{\PBR}(\htt|\varphi)$ is strictly increasing in
    $\varphi$ for $\varphi\leq\htt$, constant in $\varphi$ for
    $\varphi\geq\htt$ and strictly decreasing in $\htt$ for
    $\htt\geq\varphi$.
  \end{itemize}
\end{theorem}

\section{Comparison of $p$-Values}
\label{s:pvalcomp}

We begin by determining the relationships between $P_{\X}$, $P_{\CH}$
and $P_{\PBR}$ more precisely. Since we are interested in small
$p$-values, it is convenient to focus on the $\log(p)$-values instead
and determine their differences to $O(1/\sqrt{n})$. Because of the
identity $-\log(P_{\CH,n}(t,\varphi))=n\kl(t|\varphi)$, we reference
all $\log(p)$-values to $-\log(P_{\CH,n})$. Here we examine the
differences for $t\geq \varphi$ determined by the following theorem:

\begin{theorem}\label{t:logp-diffs}
  For $0<\varphi< \htt<1$,
  \begin{align}
    -\log(P_{\PBR,n}(\htt|\varphi)) &=
    -\log(P_{\CH,n}(\htt|\varphi)) -\frac{1}{2}\log(n+1) + \frac{1}{2}\log(2\pi\htt(1-\htt))
    + O\left(\frac{1}{n}\right), \label{e:t:logp-diffs:pbr}\\
    -\log(P_{\X,n}(t|\varphi)) &=
    -\log(P_{\CH,n}(t|\varphi))+\frac{1}{2}\log(n)
    -\log\left(\sqrt{\frac{\htt}{2\pi(1-\htt)}}\frac{1-\varphi}{\htt-\varphi}\right)
    + O\left(\frac{1}{n}\right).
    \label{e:t:logp-diffs:x}
  \end{align}
\end{theorem}

The theorem follows from Thms.~\ref{t:pbr-ch},~\ref{t:x-pbr} and
Cor.~\ref{cor:x-ch} proven in the Appendix, where explicit interval
expressions are obtained for these $\log(p)$-value differences.  The
order notation assumes fixed $t>\varphi$. The bounds are not uniform,
see the expressions in the appendix for details.

The most notable observation is that there are systematic gaps of
$\log(n)/2+O(1)$ between the $\log(p)$-values. As we already knew,
there is no question that the exact test is the best of the three for
this simple application.  While these gaps may seem large on an
absolute scale, representing factors close to $\sqrt{n}$, they are in
fact much smaller than the experiment-to-experiment variation of the
$p$-values.  To determine this variation, we consider the asymptotic
distributions.  We can readily determine that the $\log(p)$-values are
asymptotically normal with standard deviations proportional to
$\sqrt{n}$, which is transferred from the variance of $\Htheta$.
Compared to these standard deviations the gaps are negligible.  The
next theorem determines the specific way in which asymptotic normality
holds.  Let $N(\mu,\sigma^{2})$ denote the normal distribution with
mean $\mu$ and variance $\sigma^{2}$.  The notation
$X_{n}\xrightarrow{D}N(\mu,\sigma^{2})$ means that $X_{n}$ converges
in distribution to the normal distribution with mean $\mu$ and
variance $\sigma^{2}$.

\begin{theorem} Assume $0<\varphi<\theta<1$.  For $P=P_{\CH,n}$,
  $P=P_{\PBR,n}$ or $P=P_{\X,n}$, the
  $\log(p)$-value $-\log(P)$ converges in distribution according to
  \begin{equation}
    \sqrt{n}(-\log(P)/n-\kl(\theta|\varphi))\xrightarrow{D}
    N(0,\sigma^{2}_{G}),
  \end{equation}
  with 
  \begin{equation*}
    \sigma_{G}^2= \theta(1-\theta)\left(\log\left(\frac{\theta}{1-\theta}\frac{1-\varphi}{\varphi}\right)\right)^2.
  \end{equation*}
\end{theorem}

The theorem is proven in the Appendix, see Thm.~\ref{thm:gain_anorm}.
For the rest of the paper, we write $P$ or $P_{n}$ for the $p$-values
of any one of the tests when it does not matter which one.

We display the behavior described in the above theorems for $n=100$
and $\theta=0.5$ in Fig.~\ref{fig:pvalue_gap_comparison}. We conclude
that the phenomena discussed above are already apparent for small
numbers of trials. For Fig.~\ref{fig:pvalue_gap_comparison}, we
computed the quantiles of the $\log(p)$-values numerically using the
formulas provided in the previous section, substituting for $t$ the
corresponding quantile of $\Htheta$ given that $\Prob(B=1)=\theta$.
To be explicit, let $t_{r,n}(\theta)$ be the $r$-quantile of $\Htheta$
defined as the minimum value $\htheta$ of $\Htheta$ satisfying
$\Prob(\Htheta\leq \htheta)\geq r$. (For simplicity we do not place
the quantile in the middle of the relevant gap in the distribution.)
For example, $t_{0.5,n}(\theta)$ is the median. Then, by the
monotonicity properties of the tests, the $r$-quantile of
$-\log(P_{n}(\Htheta|\varphi))$ is given by
$-\log(P_{n}(t_{r,n}(\theta)|\varphi))$.

\begin{figure}
  \begin{center}
  \includegraphics[scale=0.7]{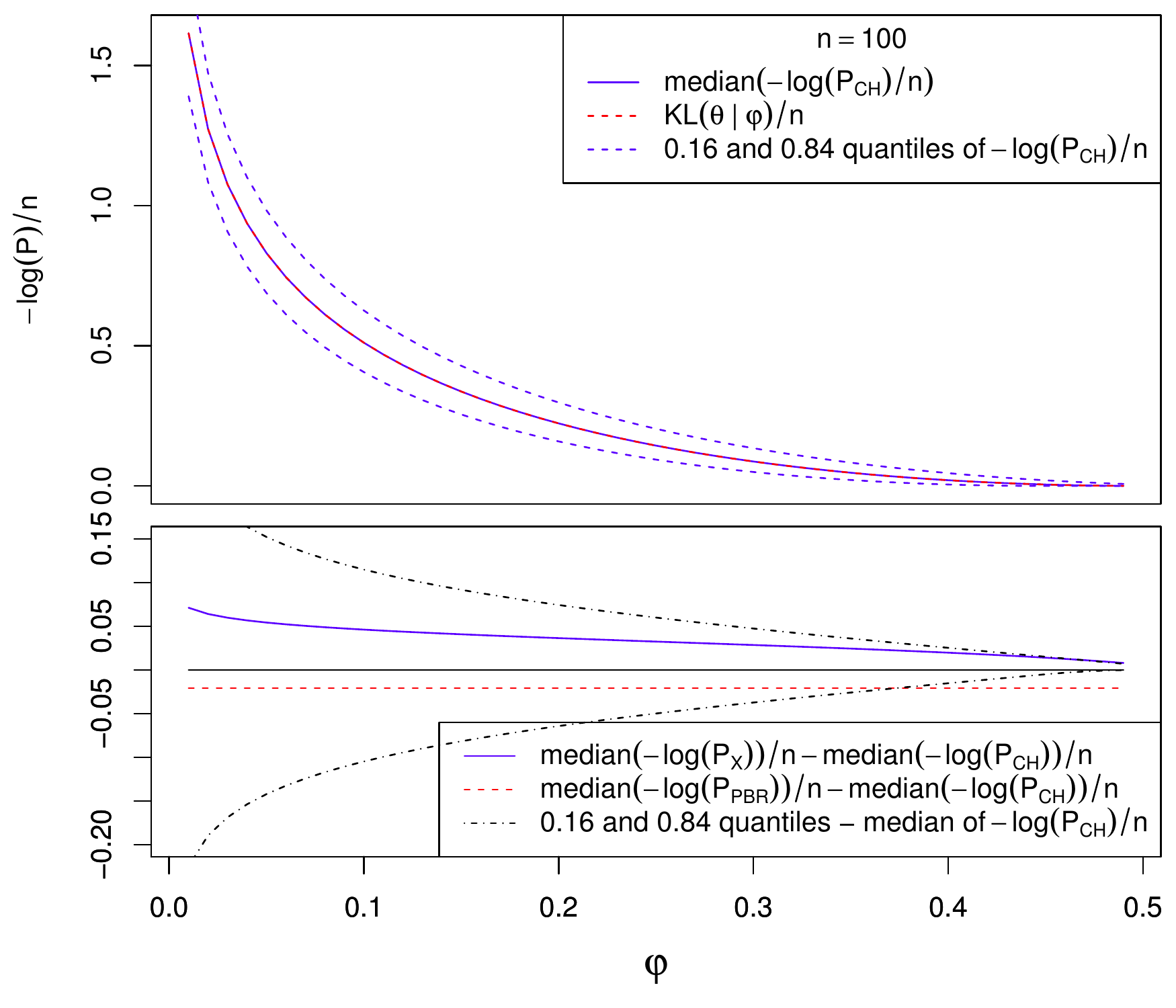}\\
  \end{center}
  \caption{Comparison of $\log(p)$-values at $n=100$ and $\theta=0.5$.
    The top half of the figure shows the median, the $0.16$ and the
    $0.84$ quantile of $-\log(P_{\CH,n}(\Htheta|\varphi))/n$.  For
    $\theta=0.5$, the median agrees with $\kl(\theta|\varphi)$ by
    symmetry.  The lower half shows the median differences
    $-\log(P(\Htheta|\varphi))/n +
    \log(P_{\CH,n}(\Htheta|\varphi))/n$ for $P=P_{\PBR,n}$ and
    $P=P_{\X,n}$. The difference between the $0.16$ and $0.84$
    quantiles and the median for
    $-\log(P_{\CH,n}(\Htheta|\varphi))/n$ are also shown where
    they are within the range of the plot; even for $n$ as small as
    $100$, they dominate the median differences, except where
    $\varphi$ approaches $\theta=0.5$, where the absolute $p$-values
    are no longer extremely small.}
    \label{fig:pvalue_gap_comparison}
\end{figure}

As noted above, the gaps between the $\log(p)$-values are of the form
$\log(n)/2+O(1)$. In fact, it is possible to determine the asymptotic
behavior of these gaps. After accounting for the explicitly given
$O(1)$ terms in Thm.~\ref{t:logp-diffs}, they are asymptotically
normal with variances of order $O(1/n)$.  The standard deviations
  of the gaps are therefore small compared to their size.  The
precise statement of their asymptotic normality is
Thm.~\ref{thm:gap_anorm} in the
Appendix. Fig.~\ref{fig:pvalue_gap_behavior} shows how these gaps
depend on the value $\htheta$ of $\Htheta$ given $\varphi$.  The
  gaps are scaled by $\log(n)$ so that they can be compared to
  $\log(n)/2$ visually for different values of $n$. The deviation from
  $\log(n)/2$ is most notable near the boundaries, where convergence
  is also slower, particularly for $P_{\X}$. This behavior is
  consistent with the divergences as $\htt$ approaches $\varphi$ in
  the explicit interval bounds in Thm.~\ref{t:x-pbr} and
  Cor.~\ref{cor:x-ch}.

\begin{figure}
  \begin{center}
  \includegraphics[scale=0.7]{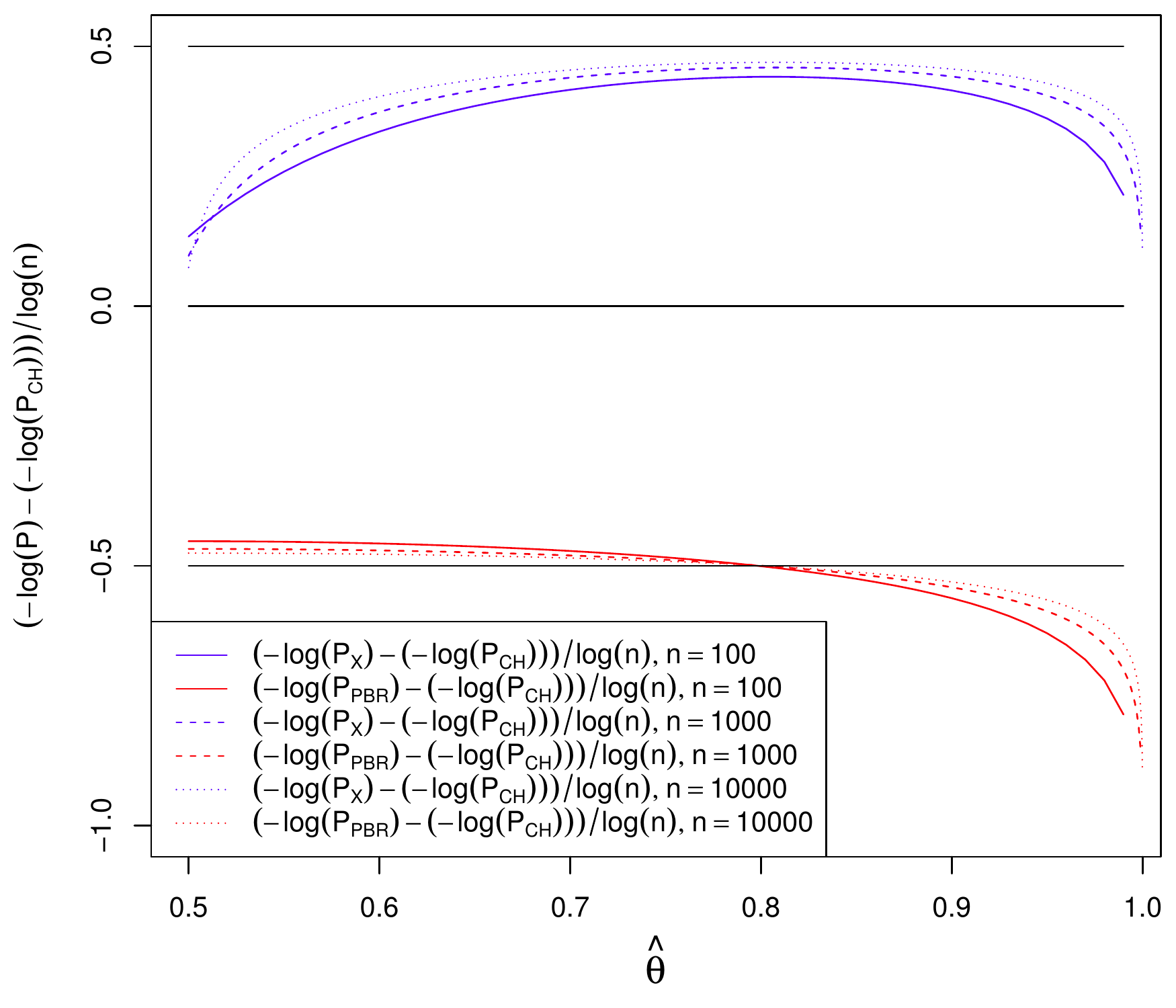}\\
  \end{center}
  \caption{Gaps between the $\log(p)$-values depending on $\htheta$ at
    $\varphi=0.5$.  We show the normalized differences
    $\left(-\log(P_{n}(\htheta|\varphi))+\log(P_{\CH,n}(\htheta|\varphi))\right)/\log(n)$
    for $P=P_{\CH}$ and $P=P_{\X}$ at $n=100,1000,10000$.  For large
    $n$, at constant $\htheta$ with $0.5<\htheta<1$, the $\PBR$ test's
    normalized difference converges to $-0.5$, and the exact test's
    converges to $0.5$. The horizontal lines at $\pm
    0.5$ indicate this limit.  The lowest order normalized asymptotic
    differences from $\pm 0.5$ are $O(1/\log(n))$ and diverge at
    $\htheta=0.5$ and $\htheta=1$.  }
  \label{fig:pvalue_gap_behavior}
\end{figure}

\section{Comparison of Confidence Intervals}
\label{s:cicomp}

Let $P$ be one of $P_{\CH,n}$, $P_{\PBR,n}$ or $P_{\X,n}$.  Given a
value $\htheta$ of $\Htheta$, the level-$a$ confidence set determined
by the test with $p$-value $P$ is $I=\{\varphi|P(\htheta|\varphi)\geq
a\}$.  By the monotonicity properties of $P$, the closure of $I$ is an
interval $[\varphi_{a}(\htheta;P),1]$. We can compute the endpoint
$\varphi_{a}$ by numerically inverting the exact expressions for
$P$. An example is shown in Fig.~\ref{fig:ci_absend}, where we show
the endpoints according to each test for $a=0.01$ and $\htheta=0.5$
as a function of $n$. All tests' endpoints converge to $0.5$ as the
number of trials grows. Notably, the relative separation between the
endpoints is not large at level $a=0.01$.

\begin{figure}
  \begin{center}
  \includegraphics[scale=0.7]{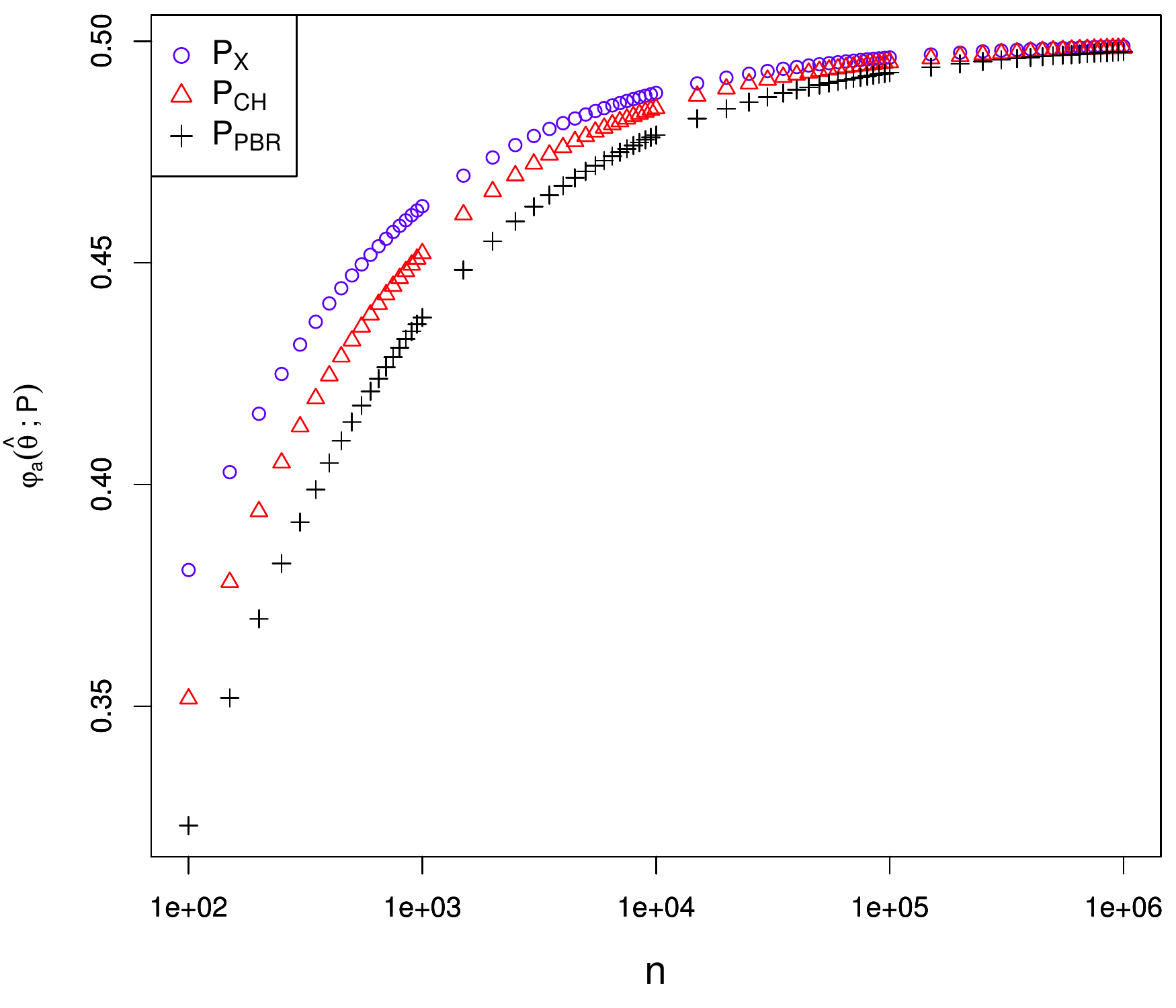}\\
  \end{center}
  \caption{Lower endpoints for the level $0.01$ confidence set of the
    three tests as a function of $n$, where
    $\htheta=0.5$.}
  \label{fig:ci_absend}
\end{figure}

To quantify the behavior of the endpoints for the different tests,
we normalize by the empirical standard deviation
$\hat\sigma=\sqrt{\htheta(1-\htheta)/n}$.
The empirical endpoint deviation is then defined as
\begin{equation}
\gamma_{a}(\htheta;P) = \frac{\htheta-\varphi_{a}(\htheta;P)}{\hat\sigma}.
\end{equation}
For the exact test and for large $n$, we expect this quantity to be
determined by the tail probabilities of a standard normal
distribution. That is, if the significance $a$ is the probability
  that a normal RV of variance $1$ exceeds $\kappa$, we expect
  $\gamma_{a}(\htheta;P_{\X})\approx \kappa$.

We take the point of view that the performance of a test is
characterized by the size of the endpoint deviation. If the relative
size of the deviations for two tests is close to $1$ then they perform
similarly for the purpose of characterizing the parameter $\theta$.
Another way of comparing the intervals obtained is to consider their
coverage probabilities. For our situation, the coverage probability
for test $P$ at $a$ can be approximated by determining $a'$ such that
$\gamma_{a'}(\theta;P_{\X})=\gamma_{a}(\theta;P)$.  From
Thm.~\ref{thm:endpts} below, one can infer that the coverage
probability is then approximately $1-a'\ge 1-a$.  The coverage
probabilities can be very conservative (larger than $1-a$),
particularly for small $a$ and $P=P_{\PBR}$.

We determined interval bounds for the empirical endpoint deviation for
all three tests. The details are in App.~\ref{app:endpoints}. The
next theorem summarizes the results asymptotically.

\begin{theorem}
  \label{thm:endpts}
  Let $q(x) = -\log(\Prob_{N(0,1)}(N\geq x))$ be the negative
  logarithm of the tail of the standard normal.  Fix
  $\htheta\in(0,1)$.  Write $\alpha=|\log(a)|$.  There is a constant
  $c$ (depending on $\htheta$) such that for $\alpha\in (1,cn]$,
  $\gamma_{a}$ satisfies 
  \begin{align}
    \gamma_{a}(\htheta;P_{\CH}) &= \sqrt{2\alpha}+O(\alpha/\sqrt{n}),\\
    \gamma_{a}(\htheta;P_{\PBR}) &= 
        \sqrt{2\alpha+\log(n)/2-\log(2\pi\htheta(1-\htheta))/2}+O(\alpha/\sqrt{n}),\\
    \gamma_{a}(\htheta;P_{\X}) &= q^{-1}(\alpha)+O(\alpha/\sqrt{n}).\label{eq:thm:endpts3}
  \end{align}
  The last expression has the following approximation relevant for
  sufficiently large $\alpha$:
  \begin{equation}
    \gamma_{a}(\htheta;P_{\X}) = 
    \sqrt{2\alpha-\log(2\pi)- \log(2\alpha-\log(2\pi))}+O(\log(\alpha)/\alpha^{3/2})+O(\alpha/\sqrt{n}).
    \label{eq:thm:endpts4}
  \end{equation}
\end{theorem}

For $\alpha=o(\sqrt{n})$, the relative error of the approximation in
the first two identities goes to zero as $n$ grows. This is not the
case for the last identity, where the relative error for large $n$ is
dominated by the term $O(\log(\alpha)/\alpha^{3/2})$, and large
$\alpha$ is required for a small relative error.

\begin{proof}
  The expression for $\gamma_{a}(\htheta;P_{\CH})$ is obtained
  from Thm.~\ref{CHinterval} in the Appendix by changing the relative
  approximation errors into absolute errors.

  To obtain the expression for $\gamma_{a}(\htheta;P_{\PBR})$, note
  that the term $\Delta$ in Thm.~\ref{thm:PBRinterval} satisfies
  $\Delta=\log(n)/2-\log(2\pi\htheta(1-\htheta))/2+O(1/n)$, see
  Thm.~\ref{t:pbr-ch}.  The $O(1/n)$ under the square root pulls out to
  an $O(1/(\sqrt{\max(\alpha,\log(n))}n))$ term that is dominated by
  $O(\alpha/\sqrt{n})$ because $\alpha\ge 1$ by assumption.

  To obtain the expressions for $\gamma_{a}(\htheta;P_{\X})$, we refer
  to Thm.~\ref{thm:exact_endpt}, where the lower bound on $\alpha$
  implies $\alpha\geq 1>\log(2)$. The intervals in
  Thm.~\ref{thm:exact_endpt} give relative errors that need to be
  converted to absolute quantities.  By positivity and monotonicity of
  $q^{-1}$, for sufficiently large $n$ and for some positive constants
  $u$ and $v$, we have
  \begin{equation}
    \gamma_{a}(\htheta;P_{\X})\in
    \left[
      q^{-1}(\alpha(1-u\sqrt{\alpha}/\sqrt{n}))(1-v\sqrt{\alpha}/\sqrt{n}),
      q^{-1}(\alpha(1+u\sqrt{\alpha}/\sqrt{n}))(1+v\sqrt{\alpha}/\sqrt{n})
    \right].
  \end{equation}
  Explicit values for $u$ and $v$ can be obtained from
  Thm.~\ref{thm:exact_endpt}.  We simplified the argument of
  $q^{-1}$ by absorbing the additive terms in the theorem into the
  term $u\alpha\sqrt{\alpha}/\sqrt{n}$ with the constant $u$ chosen
  to be sufficiently large.  Consider Eq.~\ref{eq:q-1relerr} with
  $\delta=u\sqrt{\alpha}/\sqrt{n}$.  For sufficiently large $n$, the
  expression in the denominator of the approximation error on the
  right-hand side exceeds a constant multiple of $\alpha$. From
  this, with some new constant $u'$,
  \begin{equation}
    \gamma_{a}(\htheta;P_{\X})\in
    \left[
      q^{-1}(\alpha)(1-u'\sqrt{\alpha}/\sqrt{n})(1-v\sqrt{\alpha}/\sqrt{n}),
      q^{-1}(\alpha)(1+u'\sqrt{\alpha}/\sqrt{n})(1+v\sqrt{\alpha}/\sqrt{n})
    \right],
  \end{equation}
  which, with order notation simplifies further to
  \begin{equation}
    \gamma_{a}(\htheta;P_{\X})=
    q^{-1}(\alpha)(1+O(\sqrt{\alpha}/\sqrt{n})).
  \end{equation}
  It now suffices to apply $q^{-1}(\alpha)=O(\sqrt{\alpha})$ (see
  the proof of Eq.~\ref{eq:thm:endpts4} below) and
  Eq.~\ref{eq:thm:endpts3} is obtained.  
    
  For Eq.~\ref{eq:thm:endpts4}, we bound $x=q^{-1}(\alpha)$, which we
  can do via bounds for $\alpha=q(x)$.  From the expression
  $q(x)=x^{2}/2+\log(2\pi)/2-\log(Y(x))=
  x^{2}/2+\log(2\pi)/2+\log(x)-\log(xY(x))$ in the statement of
  Thm.~\ref{thm:exact_endpt} and the bounds in Eq.~\ref{e:BndY}, we
  have the two inequalities
  \begin{align}
    q(x)&=x^{2}/2+\log(2\pi)/2 +\log(x)-\log(x Y(x))\geq x^{2}/2+\log(2\pi)/2+\log(x),\label{eq:qlowerl}\\
    q(x)&\leq  x^{2}/2+\log(2\pi)/2+\log(x)+1/x^{2}.\label{eq:qupperl}
  \end{align}
  Let $l(x)=x^{2}/2+\log(2\pi)/2+\log(x)$, which is monotonically
  increasing, as is $q$. The first inequality implies that $q^{-1}\leq
  l^{-1}$.  We need a bound of the form $q(x)\leq dx^{2}$, from which
  we can conclude that $x^{2}\geq \alpha/d$.  A bound of this type can
  be obtained from Eq.~\ref{e:q-1_bounds} in the Appendix. For
  definiteness, we restrict to $\alpha\geq 6$ and show that the bound
  holds with $d=1$.  By Eq.~\ref{eq:qupperl}, it suffices to establish
  that for $x\geq\sqrt{6}$, $l(x)+1/x^{2}\leq x^{2}$.  Since
  $\log(2\pi)/2\leq 1$, we have $\log(2\pi)/2+\log(x)+1/x^{2}\leq
  1+\log(1+(x-1))+1/x^{2} \leq x+1/x^{2}$. For $x\geq 9/4$,
  $x+1/x^{2}\leq x^{2}/2$. To finish the argument, apply the
  inequality $\sqrt{6}\geq 9/4$. \MKc{Permanent comment to check this
    inequality: We need $(1+1/x^{3})\leq x/2$. The left-hand side is
    decreasing and the right-hand side is increasing, so it suffices
    to check the boundary value. Setting $x=2+1/4$ works since
    $1/x^{3}\leq 1/8$.  Also $(9/4)^{2}=81/16\leq 6$.  The context
    does not require that we obtain tighter bounds.}

  Given the bound $x^{2}\geq\alpha$, Eq.~\ref{eq:qupperl} becomes
  $q(x)= \alpha\leq l(x)+1/\alpha$. With
  Eq.~\ref{eq:qlowerl} we get $\alpha=q(x) \in
  l(x)+[0,1]/\alpha$. Equivalently,
  \begin{equation}
     l(x)\in \alpha+\frac{1}{\alpha}[-1,0].
  \end{equation}
  Applying the monotone $l^{-1}$ on both sides gives
  \begin{equation}
     x=l^{-1}(l(x))\in l^{-1}\left(\alpha+\frac{1}{\alpha}[-1,0]\right).
  \end{equation}
  Let $\alpha'$ satisfy $x=l^{-1}(\alpha')$ with
  $\alpha'=\alpha+\delta$ and $\delta\in[-1,0]/\alpha$. Write
  $z=x^{2}$ and $c=\log(2\pi)$.  We have
  $l(z^{1/2})=z/2+c/2+\log(z)/2=\alpha'$, which we can write as a
  fixed point equation $z=f(z)$ for $z$ with
  $f(z)=2\alpha'-c-\log(z)$.  We can accomplish our goal by
  determining lower and upper bounds on the fixed point.  Since
  $\frac{d}{dy}f(y)=-1/y<0$ for $y>0$, the iteration
  $z_{0}=2\alpha'-c$ and $z_{k}=f(z_{k-1})$ is alternating around the
  fixed point $z$, provided $z_{k}>0$ for all $k$.  Provided
  $z_{0}>1$, $z_{1}=f(z_{0})<z_{0}$, from which we conclude that
  $z_{1}\leq z<z_{0}$. Since we are assuming that $\alpha\geq 6$ and
  from above $z_{0}\geq 2(\alpha-1/\alpha)-c$, the condition $z_{0}>1$
  is satisfied.  If $z_{1}\geq 1$, then $0>\frac{d}{dy}f(y)>-1$
  between $z_{1}$ and $z_{0}$, which implies that $z_{0}$ and $z_{1}$
  are in the region where the iteration converges to $z$. For our
  bounds, we only require $z_{1}>0$, so that we can bound $z$
  according to $z_{1}\leq z\leq z_{2}$. That $z_{1}>0$ follows from
  $\log(y)<y$ for $y>0$.  We have
  \begin{align}
     z_{2}-z_{1} &= z_{0}-\log(z_{1}) - (z_{0}-\log(z_{0})) \notag\\
     &= \log(z_{0}/z_{1})\notag\\
     &= \log(z_{0}/(z_{0}-\log(z_{0})))\notag\\
     &= -\log(1-\log(z_{0})/z_{0})\notag\\
     &= O(\log(z_{0})/z_{0}) = O(\log(\alpha')/\alpha')= O(\log(\alpha)/\alpha),
     \label{eq:z2-z1}
  \end{align}
  where $z_{0}=2\alpha'-c\in 2\alpha-c + 2[-1,0]/\alpha$,
  and so $-\log(z_0)=-\log(2\alpha-c)+O(1/\alpha^{2})$.
  For $z_{1}$ we get $z_{1}=z_{0}-\log(z_{0}) = 2\alpha-c-\log(2\alpha-c)+O(1/\alpha)$.
  Applying Eq.~\ref{eq:z2-z1} and from the definitions,
  \begin{equation}
    q^{-1}(\alpha) = x = \sqrt{2\alpha-c-\log(2\alpha-c)+O(\log(\alpha)/\alpha)}.
  \end{equation}  
  The approximation error in Eq.~\ref{eq:thm:endpts4} is obtained by
  expanding the square root.  We could have used Newton's method
  starting from $z_{0}$ to obtain better approximations in one step,
  but the resulting expression is more involved.
\end{proof}

The expression for $\gamma_{a}(\htheta;P_{\X})$ confirms our
expectation that it approaches the expected value for a standard
normal distribution and may be compared to the Berry-Esseen
theorem~\cite{Nagaev}. The empirical endpoint deviation of the $\CH$
test approaches that of the exact test for small $a$ (large $\alpha$).
Their squares differ by a term of order $\log(\alpha)=\log|\log(a)|$.
Notably, the ratio of the $\PBR$ and $\CH$ tests' empirical endpoint
deviation grows as $\Theta(\sqrt{\log(n)/\alpha})$. The relationships
are visualized in
Figs.~\ref{fig:normalized_endpts1},~\ref{fig:normalized_endpts2}
and~\ref{fig:normalized_endpts3} for different values of $a$.  The
figures show that the relative sizes of the empirical endpoint
deviations tend toward $1$ with smaller $a$. The
$\Theta(\sqrt{\log(n)/\alpha})$ relative growth of the $\PBR$ test's
endpoint deviations leads to less than a doubling of the deviations
relative to the exact test's at $a=0.01$ and $a=0.001$ even for
$n=10^{6}$. So while the test's coverage probabilities are much closer
to $1$ than the nominal value of $1-a$, we believe that it does not
lead to unreasonably conservative results in many applications.

\begin{figure}
  \begin{center}
  \includegraphics[scale=0.7]{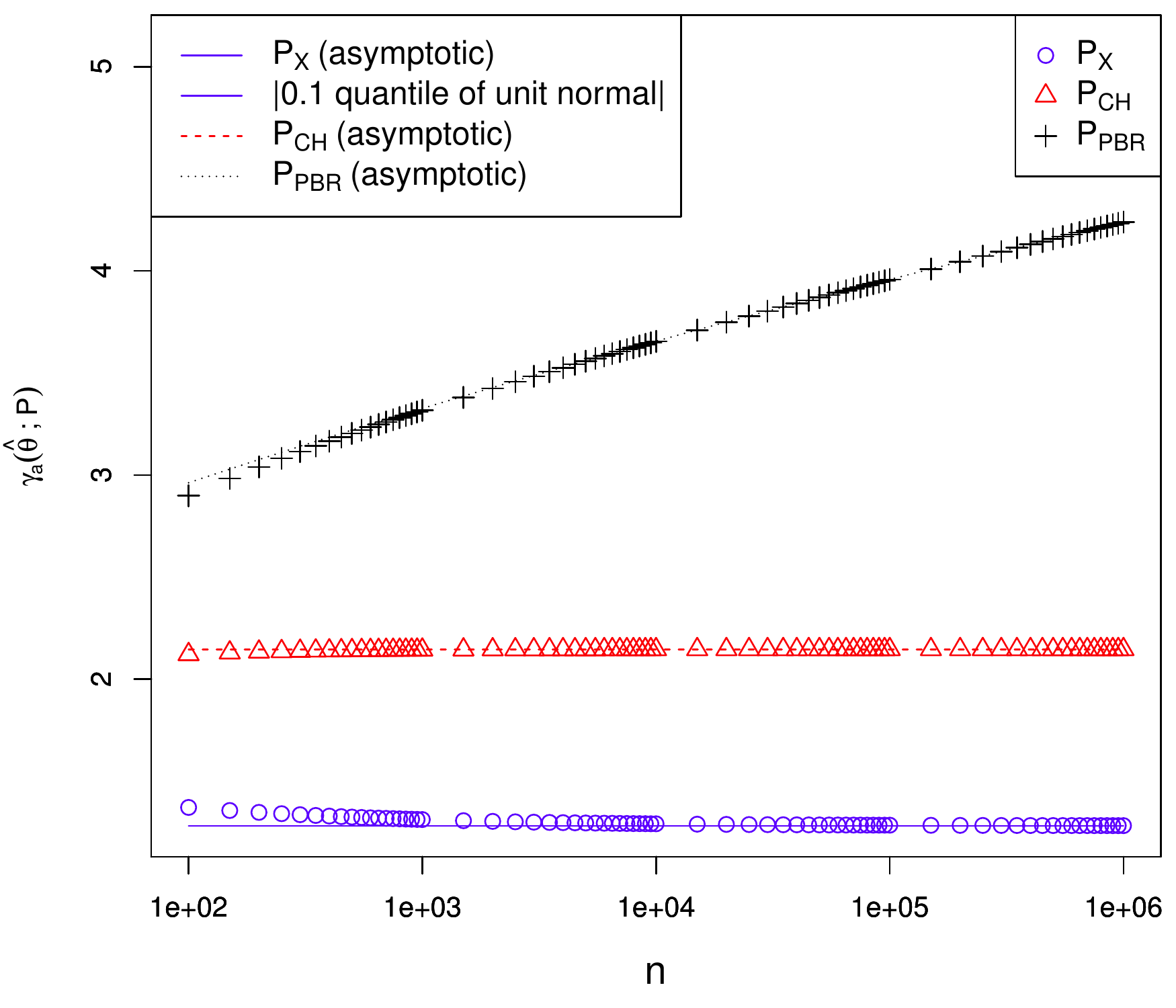}\\
  \end{center}
  \caption{Empirical confidence set endpoint deviations at level
    $a=0.1$ for $\htheta=0.5$ as a function of $n$.  The continuous
    lines show the expressions obtained after dropping the
    $O(1/\sqrt{n})$ terms. For the exact test, these expressions are
    the same as the normal approximation and therefore match the
    absolute value of the $0.1$ quantile of a unit normal.}
  \label{fig:normalized_endpts1}
\end{figure}

\begin{figure}
  \begin{center}
  \includegraphics[scale=0.7]{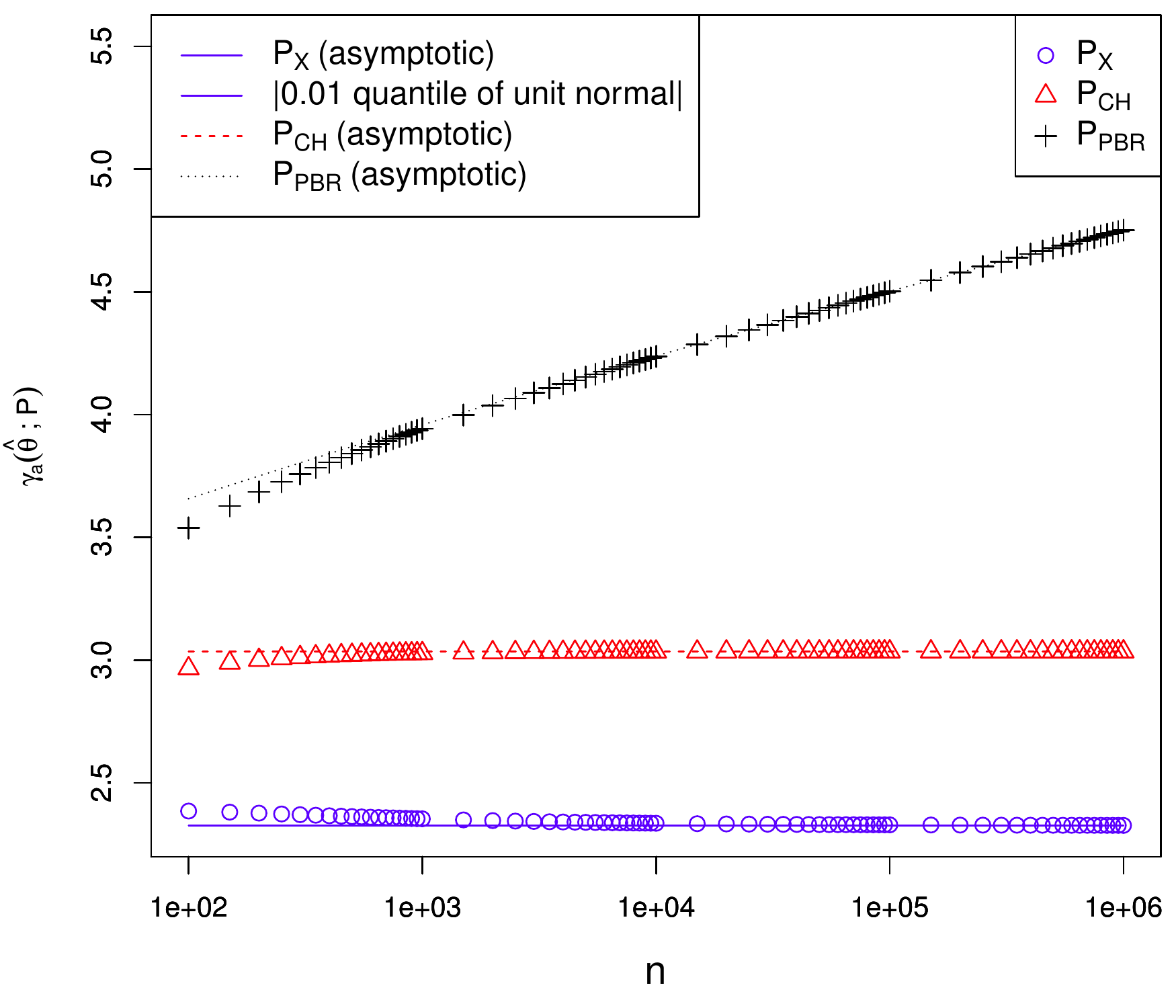}\\
  \end{center}
  \caption{Empirical confidence set endpoint deviations at level
    $a=0.01$ for $\htheta=0.5$ as a function of $n$.  See the
    caption of Fig.~\ref{fig:normalized_endpts1}.}
  \label{fig:normalized_endpts2}
\end{figure}

\begin{figure}
  \begin{center}
  \includegraphics[scale=0.7]{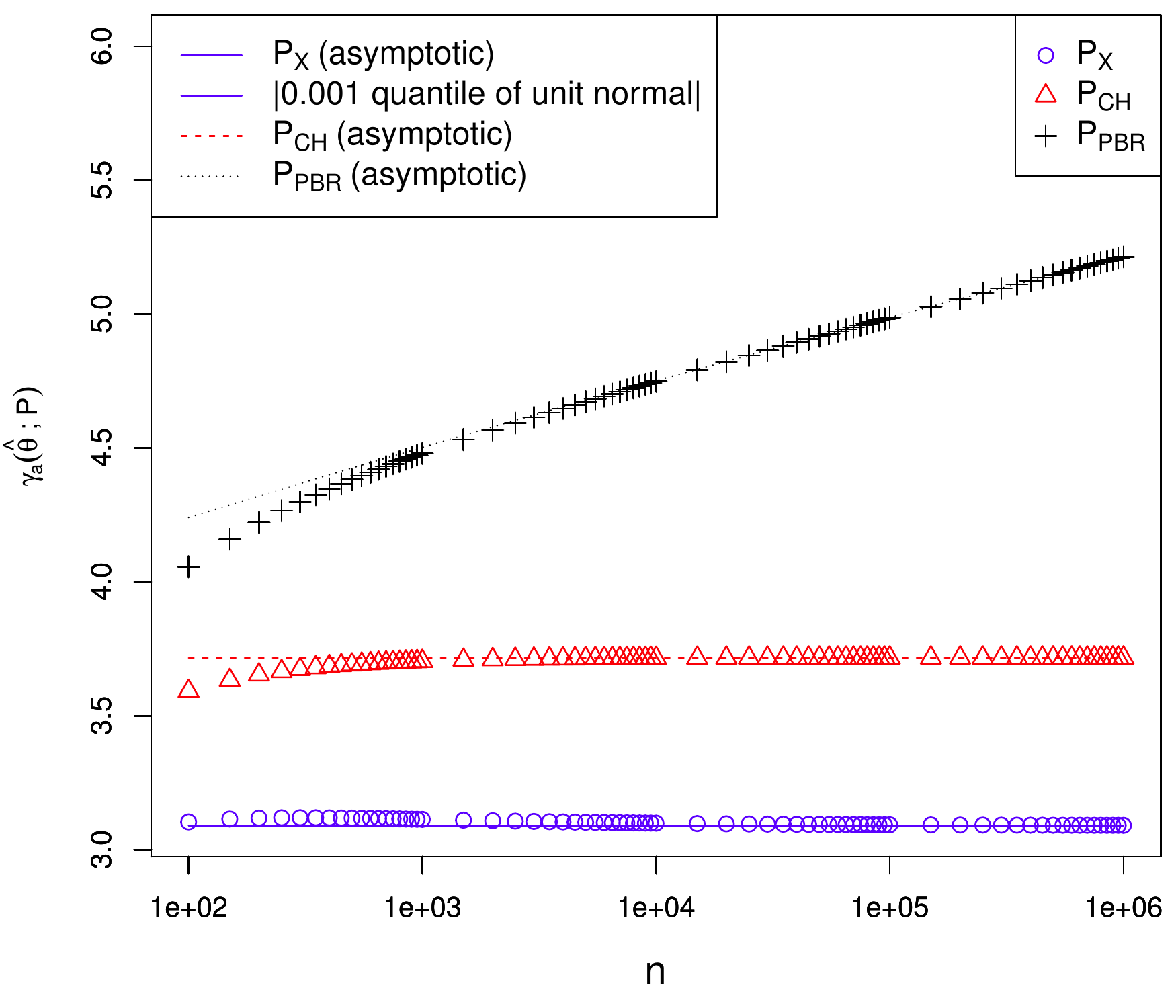}\\
  \end{center}
  \caption{Empirical confidence set endpoint deviations at level
    $a=0.001$ for $\htheta=0.5$ as a function of $n$.  See the
    caption of Fig.~\ref{fig:normalized_endpts1}.}
  \label{fig:normalized_endpts3}
\end{figure}

Next we consider the behavior of the true
endpoint deviations given by the normalized difference
of the true success probability $\theta$ and the endpoint obtained
from one of the tests.  Let $\sigma=\sqrt{\theta(1-\theta)/n}$ be
the true standard deviation and define the true endpoint deviation
determined by test $P$ as
\begin{align}
  \tilde\gamma_{a}(\Htheta|P) &= (\theta-\varphi_{a}(\Htheta|P))/\sigma\notag\\
  &= (\theta-\Htheta)/\sigma + \gamma_{a}(\Htheta|P)\hat\sigma/\sigma.
\end{align}
The true endpoint deviations show how the inferred endpoint compares
to $\theta$ and therefore directly exhibits the statistical
fluctuations of $\Htheta$. In contrast, the empirical endpoint
deviations are to lowest order independent of $\htheta-\theta$. 

We take the view that two tests' endpoints perform similarly if their
true endpoint deviations differ by an amount that is small compared to
the width of the distribution of the true endpoint deviations. To
compare the three tests on this basis, we consider the quantiles for
$\Htheta$ corresponding to $\pm\kappa$ Gaussian standard deviations
from $\theta$ with $\kappa$ constant.  The quantiles satisfy
$\theta_{\pm \kappa}=\theta\pm\kappa\sigma(1+O(1/\sqrt{n}))$, by the
Berry-Esseen theorem or from Thm.~\ref{thm:exact_endpt}. Since
$\hat\sigma=\sigma(1+O(1/\sqrt{n}))$, we can also see that
$\gamma_{a}(\theta_{\pm
  \kappa}|P)=\gamma_{a}(\htheta|P)+O(1/\sqrt{n})$, and so by
  substituting into the definition,
\begin{align}
  \tilde\gamma_{a}(\theta_{\pm \kappa}|P) = 
  \gamma_{a}(\theta|P)\pm \kappa +O(1/\sqrt{n}),
\end{align}
where the implicit constants depend on $\kappa$.  For large $\alpha$,
the $\CH$ and exact tests' endpoints are close and are dominated by
$\kappa$, so they perform similarly. But this does not hold for the
comparison of the $\CH$ or the exact test's endpoints to those of the
$\PBR$ test, since the latter's endpoint deviation grows as
$\sqrt{\log(n)/2}$.

The $\PBR$ test's robustness to stopping rules requires that endpoint
deviations must grow. Qualitatively, we expect a growth of at least
$\Omega(\sqrt{\log\log(n)})$ due to the law of the iterated
logarithm. This growth is slower than the $\sqrt{\log(n)/2}$ growth
found above, suggesting that improvements are possible, as observed in
Ref.~\cite{Shafer}. In many applications, the number of trials to be
acquired can be determined ahead of time, so full robustness to
stopping rules is not necessary. However, the ability to adapt to
changing experimental conditions may still be helpful, as the example
in Sect.~\ref{s:basic} shows. If we know the number of trials ahead of
time, we can retain the ability to adapt while avoiding the asymptotic
growth of the endpoint deviations of the $\PBR$ test.

A strategy for avoiding the asymptotic growth of the $\PBR$ test's
endpoint deviations is to set aside the first $m=\lambda n$ of the
trials for training to infer the probability of success, and then use
this to determine the test factor to be used on the remaining
$(1-\lambda)n$ of the trials.  With this strategy, the endpoint
deviations are bounded on average and typically. We formalize the
training strategy as follows: Modify Eq.~\ref{eq:pbr_testfacts} by
setting $F_{k=1}=1$ for $k< m$ and for $k\ge m$,
\begin{equation}
  F_{k+1}(B_{k+1}) = F(B_{k+1})=\left(\frac{\Htheta_{m}}{\varphi}\right)^{B_{k+1}}
  \left(\frac{1-\Htheta_{m}}{1-\varphi}\right)^{1-B_{k+1}}.
\end{equation}
Let $G=F$ if $\varphi\leq\Htheta_{m}$ and $G=1$ otherwise.  The
$G_{k+1}$ are valid test factors for the null $\cB_{\varphi}$.  A
$p$-value for testing $\overline{\cB}_{\varphi}$ is given by
\begin{align}
  P_{\lambda}(\mathbf{B}|\varphi) &= G(1)^{-(n-m)\Htheta'_{m}}G(0)^{-(n-m)(1-\Htheta'_{m})}
\end{align}
where $\Htheta'_{m}$ is defined by
$(n-m)\Htheta'_{m}+m\Htheta_{m}=n\Htheta_{n}$.  We call this the
$P_{\lambda}$ test.

Define
\begin{align}
  Q_{\lambda}(\mathbf{B}|\varphi) &=
  \left(\frac{\varphi}{\Htheta_{m}}\right)^{(n-m)\Htheta'_{m}}\left(\frac{1-\varphi}{1-\Htheta_{m}}\right)^{(n-m)(1-\Htheta'_{m})}.
\end{align}
Then for $\varphi\leq\Htheta_{m}$,
$Q_{\lambda}(\mathbf{B}|\varphi)=P_{\lambda}(\mathbf{B}|\varphi)$.
To investigate the behavior of these quantities, we consider values
$\mathbf{b}$, $\htheta$, $\htheta_{m}$ and $\htheta'_{m}$ of the
corresponding RVs. As a function of $\varphi$,
$Q_{\lambda}(\mathbf{b}|\varphi)$ is maximized at
$\varphi=\htheta'_{m}$ and monotone on either side of
$\htheta'_{m}$. If $\htheta_{m}\leq\varphi\leq \htheta'_{m}$, then
$Q_{\lambda}(\mathbf{b}|\varphi)\geq
1=P_{\lambda}(\mathbf{b}|\varphi)$, So for
$\varphi\leq\max(\htheta_{m},\htheta'_{m})$, we can use
$Q_{\lambda}$ instead of $P_{\lambda}$ without changing endpoint
calculations.

For determining the endpoint of a level-$a$ one-sided confidence
interval from $P_{\lambda}$, we seek the maximum $\varphi$ such that
for all $\varphi'\leq\varphi$, $P_{\lambda}(\mathbf{b}|\varphi')\leq
a$.  This maximum value of $\varphi$ satisfies that
$\varphi\leq\min(\htheta'_{m},\htheta_{m})$: For
$\htheta_{m}\leq\htheta'_{m}$, this follows from
$P_{\lambda}(\mathbf{b}|\htheta_{m})=1$.  For $\htheta_{m}\geq
\htheta'_{m}$, the location of the maximum of $Q_{\lambda}$ implies
that $P_{\lambda}(\mathbf{b}|\htheta'_{m})\geq
P_{\lambda}(\mathbf{b}|\htheta_{m})=1$.  

We show that endpoint deviations from the $P_{\lambda}$ test are
typically a constant factor larger than those of the $\CH$ test. For
large $\alpha$, the factor approaches $1/\sqrt{1-\lambda}$,
approximating the endpoint deviations for a $\CH$ test with
$(1-\lambda)n$ trials. We begin by comparing
  $P_{\lambda}$ to $P_{\CH,(1-\lambda)n}$ with the latter applied to the
  last $(1-\lambda)n$ trials and $\varphi\leq\htheta'_{m}$,
  where we can use $Q_{\lambda}$ in place of $P_{\lambda}$.
\begin{equation}
  Q_{\lambda}(\mathbf{b}|\varphi)/P_{\CH,(1-\lambda)n}(\htheta'_{m}|\varphi)
  = \left(\frac{\htheta'_{m}}{\htheta_{m}}\right)^{(1-\lambda)n\htheta'_{m}}
  \left(\frac{1-\htheta'_{m}}{1-\htheta_{m}}\right)^{(1-\lambda)n(1-\htheta'_{m})}.
\end{equation}
Or, for the $\log(p)$-value difference $l_{p}$,
\begin{equation}
  l_{p} = 
   -\log(Q_{\lambda}(\mathbf{b}|\varphi))+\log(P_{\CH,(1-\lambda)n}(\htheta'_{m}|\varphi))
  = -(1-\lambda)n\kl(\htheta'_{m}|\htheta_{m}),
\end{equation}
which is non-positive.  By expanding to second order,
\begin{align}
  \kl(\htt+x|\htt+y) &=
   (\htt+x)\left(\log(1+x/\htt)-\log(1+y/\htt)\right)
   \notag\\
   &\hphantom{=\;\;}
       + (1-\htt-x)\left(
         \log(1-x/(1-\htt))-\log(1-y/(1-\htt))\right)\notag\\
   &=
   \frac{(x-y)^{2}}{2\htt(1-\htt)}
   + O(\max(|x|,|y|)^{3}).
\end{align}
Let $\Delta = \Htheta_{m}-\theta$ and $\Delta'=\Htheta'_{m}-\theta$. From
the above expansion with $\htt=\theta$, $x=\delta'$ and $y=\delta$ (where
$\delta$ and $\delta'$ are values of $\Delta$ and $\Delta'$)
\begin{align}
  l_{p}&= -(1-\lambda)n\left(
    \frac{(\delta-\delta')^{2}}{2\theta(1-\theta)}+O(\max(|\delta|,|\delta'|^{3}))\right).
\end{align}
The RVs $\Delta$ and $\Delta'$ are independent with means $0$ and
variances $\sigma^{2}/\lambda$ and
$\sigma^{2}/(1-\lambda)$. Furthermore, $\sqrt{n}\Delta$ and
$\sqrt{n}\Delta'$ are asymptotically normal with variances
$\theta(1-\theta)/\lambda$ and $\theta(1-\theta)/(1-\lambda)$.
Consequently, the RV $\sqrt{n}(\Delta-\Delta')$ is asymptotically
normal with variance
$v=\theta(1-\theta)/(\lambda(1-\lambda))$. Accordingly, the
probability that
$n(\Delta-\Delta')^{2}\ge\kappa^{2}\theta(1-\theta)/(\lambda(1-\lambda))$
is asymptotically given by the two-sided tail for $\kappa$ standard
deviations of the standard normal.  For determining typical behavior,
we consider
$(\delta-\delta')^{2}=\kappa^{2}\theta(1-\theta)/(n\lambda(1-\lambda))$
with $\kappa\geq 0$ constant for asymptotic purposes.  Observe that
$n\Delta^{3}$ and $n\Delta'^{3}$ are $\tilde O(1/\sqrt{n})$ with
probability $1$, where the ``soft-O'' notation $\tilde O$ subsumes the
polylogarithmic factor from the law of the iterated logarithm. We can
now write
\begin{equation}
   l_{p} = -\frac{\kappa^{2}}{2\lambda} + \tilde O(1/\sqrt{n}).
\end{equation}
Fix the level $a$ and thereby also $\alpha=|\log(a)|$.  Define
$\hat\sigma'=\sqrt{\htheta'_{m}(1-\htheta'_{m})/(1-\lambda)n}$, and
let $\varphi'=\htheta'_{m}-\gamma'\hat\sigma'$ be the smallest
solution of $-\log(Q_{\lambda}(\mathbf{b}|\varphi'))=\alpha$.  Because
\begin{equation}
-\log(Q_{\lambda}(\mathbf{b}|\varphi')) = 
  -\log(P_{\CH,(1-\lambda)n}(\htheta'_{m}|\varphi'))+l_{p},
\end{equation}
we can estimate $\gamma'$ as
$\gamma'=\gamma_{a',(1-\lambda)n}(\htheta'_{m};P_{\CH})=\sqrt{2(\alpha-l_{p})}+O(\alpha/\sqrt{n})$
with $a'=e^{-(\alpha-l_{p})}$. Here, the subscript $(1-\lambda)n$ of
$\gamma_{a'}$ makes the previously implicit number of trials explicit.

To finish our comparison, we express the endpoint $\varphi'$ relative to $\htheta$.
For this, we write
\begin{align}
\varphi' &=\htheta'_{m}-\gamma'\hat\sigma'\notag\\
  &=\htheta + (\htheta'_{m}-\htheta)
  -\gamma'\hat\sigma\sqrt{\frac{\htheta'_{m}(1-\htheta'_{m})}{(1-\lambda)\htheta(1-\htheta)}}\notag\\
  &= \htheta + (\htheta'_{m}-\htheta)
  -\frac{\gamma'}{\sqrt{1-\lambda}}\hat\sigma\left(1 + O(|\htheta-\htheta'_{m}|)\right).
\end{align}
We have $\htheta'_{m}-\htheta=
\lambda(\htheta'_{m}-\htheta_{m})=\lambda(\delta'-\delta)$, and we
are considering the case
$\lambda|\delta'-\delta|=\kappa\sqrt{\lambda\theta(1-\theta)/(n(1-\lambda))}$, so
\begin{equation}
  \varphi' = \htheta-\frac{\gamma'}{\sqrt{1-\lambda}}\hat\sigma\left(1 + O(1/\sqrt{n})\right).
\end{equation}
We can therefore identify
\begin{align}
  \gamma_{a}(\htheta|P_{\lambda})
  &=\frac{\gamma'}{\sqrt{1-\lambda}}(1 + O(1/\sqrt{n}))\notag\\
  &=\frac{\sqrt{2(\alpha+\kappa^{2}/(2\lambda)+\tilde O(1/\sqrt{n}))}+O(\alpha/\sqrt{n})}{\sqrt{1-\lambda}}(1 + O(1/\sqrt{n}))\notag\\
  &=\frac{\sqrt{2(\alpha+\kappa^{2}/(2\lambda))}}{\sqrt{1-\lambda}} + \tilde O(\alpha/\sqrt{n}),
\end{align}
which compares as promised to $\gamma_{a}(\htheta;P_{\CH})=\sqrt{2\alpha}+O(\alpha/\sqrt{n})$.

\section{Conclusion}
\label{s:conclusion}

It is clear that for the specific problem of one-sided hypothesis
testing and confidence intervals for Bernoulli RVs, it is always
preferable to use the exact test in the ideal case, where the
  trials are i.i.d. For general nulls, exact tests are typically not
available, so approximations are used. The approximations often
  do not take into account failure of underlying distributional
  assumptions. The approximation errors can be large at high
significance. Thus trustworthy alternatives such as those based on
large deviation bounds or test supermartingales are desirable. Our
goal here is not to suggest that these alternatives are better for the
example of Bernoulli RVs, but to determine the gap between them and an
exact test, in a case where an exact test is known and all tests are
readily calculable. The suggestion is that for high significance
applications, the gaps are relatively small on the relevant
logarithmic scale.  For $p$-values, they are within what is expected
from experiment-to-experiment variation, even for moderate
significances. For confidence intervals, the increase in size is
bounded by a constant if the number of trials is known ahead of time,
but there is a slowly growing cost with number of trials if we allow
for arbitrary stopping-rules.


\section{Appendix}
\label{s:app}

\subsection{Preliminaries}

Notation and definitions are as introduced in the text. The
$p$-value bounds obtained by the three tests investigated are
denoted by $P_{\X}$ for the exact, $P_{\CH}$ for the
Chernoff-Hoeffding, and $P_{\PBR}$ for the PBR test. They depend on
$n$, $\varphi$ and $\Htheta$. For reference, here are the
definitions again. 
\begin{align}
  P_{\X}(\Htheta|\varphi,n) &= \sum_{i=n\Htheta}^n
  \varphi^i(1-\varphi)^{n-i}
  \binom{n}{i},\notag\\
  P_{\CH}(\Htheta|\varphi) 
  &= \left\{ \begin{array}{ll}
      \left(\frac{\varphi}{\Htheta}\right)^{n\Htheta}\left(\frac{1-\varphi}{1-\Htheta}\right)^{n(1-\Htheta)} & \textrm{if $\Htheta\geq\varphi$,}\\
      1 &\textrm{otherwise.}
    \end{array}
  \right.  \notag\\
  P_{\PBR}(\Htheta|\varphi) &= \left\{ \begin{array}{ll}
      \varphi^{n\Htheta}(1-\varphi)^{n(1-\Htheta)}(n+1)\binom{n}{n\Htheta}
      & \textrm{if $\Htheta\geq\varphi$,}\\
      \Htheta^{n\Htheta}(1-\Htheta)^{n(1-\Htheta)}(n+1)\binom{n}{n\Htheta}&\textrm{otherwise.}
    \end{array}
  \right.
\end{align}
The gain per trial for a $p$-value bound $P_{n}$ is
$G_{n}(P_{n})=-\log(P_{n})/n$.  The values of $\varphi$, $\hat\theta$
and $\theta$ are usually constrained.  Unless otherwise stated, we assume that
$0<\varphi,\htheta,\theta<1$ and $n\geq 1$.

Most of this appendix is dedicated to obtaining upper and lower bounds
on $\log(p)$-values and lower bounds on endpoints of confidence
intervals.  We make sure that the upper and lower bounds differ by
quantities that converge to zero as $n$ grows. Their differences are
$O(1/n)$ for $\log(p)$-values and $O(1/\sqrt{n})$ for confidence lower
bounds.  We generally aim for simplicity when expressing these bounds,
so we do not obtain tight constants.

\subsection{Closed-Form Expression for $P_{\PBR}$}
\label{a:pbrexpression}

\begin{theorem} \label{t:pbrexpression}
  Define
  \begin{align}
    \widetilde{\Theta}_k &= \frac{1}{k+2}\left(S_{k} + 1\right),\notag\\
    F_{k+1} &= \left(\frac{\widetilde{\Theta}_{k}}{\varphi}\right)^{B_{k+1}}
    \left(\frac{1-\widetilde{\Theta}_{k}}{1-\varphi}\right)^{1-B_{k+1}}.
  \end{align}
  Then
  \begin{equation}
     \frac{1}{\prod_{k=1}^{n} F_{k}} = \varphi^{n\Htheta}(1-\varphi)^{n(1-\Htheta)}(n+1)\binom{n}{n\Htheta}.
    \label{e:pbrformula}
  \end{equation}
\end{theorem}

\begin{proof}
  The proof proceeds by induction. Write $P_{k}$ for the right-hand
  side of Eq.~\ref{e:pbrformula}.  For $n=0$, $P_{0}=1$, and the
  left-hand side of Eq.~\ref{e:pbrformula} evaluates to $1$ as
  required, with the usual convention that the empty product evaluates
  to $1$.

  Now suppose that Eq.~\ref{e:pbrformula} holds at trial $n=k$.  For
  $n=k+1$ we can use $(k+1)\Htheta_{k+1}=S_{k+1}=S_{k}+B_{k+1}$. 
  We expand the binomial expression to rewrite the right-hand side as
  \begin{align}
    P_{k+1}&=\varphi^{k\Htheta_{k}+B_{k+1}}(1-\varphi)^{k(1-\Htheta_{k})+(1-B_{k+1})}(k+2)
    \binom{k+1}{k\Htheta_{k}+B_{k+1}} \notag\\
    &=
    \varphi^{k\Htheta_{k}}(1-\varphi)^{k(1-\Htheta_{k})} (k+1)
    \binom{k}{k\Htheta_{k}} \notag\\
    &\hphantom{=\;\;}\cdot \varphi^{B_{k+1}}(1-\varphi)^{1-B_{k+1}}
    (k+2)(k-k\Htheta_{k}+1)^{-(1-B_{k+1})}(k\Htheta_{k}+1)^{-B_{k+1}}.
  \end{align}
  Since
  $\widetilde{\Theta}_{k}=(S_{k}+1)/(k+2)=(k\Htheta_{k}+1)/(k+2)$ and
  $1-\widetilde{\Theta}_{k}=(k-S_{k}+1)/(k+2)=(k-k\Htheta_{k}+1)/(k+2)$,
  the identity simplifies to
  \begin{equation}
    P_{k+1} = P_{k}\cdot \frac{1}{F_{k+1}},
  \end{equation}
  thus establishing the induction step.
\end{proof}

The expression in Eq.~\ref{e:pbrformula} can be seen as the inverse of
a positive martingale for $\cH_{0}=\{\nu_{\varphi}\}$ determined by $S_{n}$. The
complete family of such martingales was obtained by
Ville~\cite{ville:1939a}, Chapter 5, Sect. 3, Eq. 21.  Ours is
obtained from Ville's with $dF(t)=dt$ as the probability measure.

\subsection{Log-$p$-Value Approximations}

We use $-\log(P_{\CH,n}(\htt|\varphi))=n\kl(\htt|\varphi)$ as our
reference value. According to Thm.~\ref{t:pproperties}, the
$\log(p)$-values are ordered according to
$-\log(P_{\PBR})\leq-\log(P_{\CH})\leq-\log(P_{\X})$.  To express the
asymptotic differences between the $\log(p)$-values, we use auxiliary
functions. The first is
\begin{align}
  H_{n}(\htt) &= -\log\left(\htt^{n\htt}(1-\htt)^{n(1-\htt)}\binom{n}{n\htt}\sqrt{n+1}\right)\notag\\
    &= -n\htt\log(\htt)-n(1-\htt)\log(1-\htt)  - \log\binom{n}{n\htt} - \frac{1}{2}\log(n+1). \label{e:Hn}
\end{align}
The first two terms of this expression can be recognized as the
Shannon entropy of $n$ independent random bits, each with probability
$\htt$ for bit value $1$.  For $\htt\in [1/n,1-1/n]$ and
with Stirling's approximation $\sqrt{2\pi n}(n/e)^{n}e^{1/(12n+1)}\leq n!\leq
\sqrt{2\pi n}(n/e)^{n}e^{1/(12n)}$ applied to the binomial
coefficient, we get 
\begin{align}
  \log\binom{n}{nt} &= \log\left(\frac{n!}{(tn)!((1-t)n)!}\right)\notag\\
    &\in
    \log\left(\frac{\sqrt{2\pi n}}{\sqrt{2\pi tn}\sqrt{2\pi (1-t)n}}\right)
     + \log\left(\frac{(n/e)^{n}}{(tn/e)^{tn}((1-t)n/e)^{(1-t)n}}\right)\notag\\
    &\hphantom{=\;\;}
     + \left[\frac{1}{12n+1},\frac{1}{12n}\right] + \left[-\frac{1}{12tn}-\frac{1}{12(1-t)n},-\frac{1}{12tn+1}-\frac{1}{12(1-t)n+1}\right]
     \notag\\
  &= -\frac{1}{2}\log(2\pi t(1-t))-\frac{1}{2}\log(n)
    -tn\log(t)-(1-t)n\log(1-t)\notag\\
    &\hphantom{=\;\;}
     + \left[\frac{1}{12n+1}-\frac{1}{12t(1-t)n},\frac{1}{12n}-
                \frac{12n+2}{(12tn+1)(12(1-t)n+1)}\right].
\end{align}
We can increase the interval to simplify the bounds while preserving
convergence for large $n$. For the lower bound, we use
$-1/(12t(1-t)n)$. For the upper bound, note that $(12tn+1)(12(1-t)n+1)$
is maximized at $t=1/2$. We can therefore increase the upper bound
according to
\begin{equation}
\frac{1}{12n}-
                \frac{12n+2}{(12tn+1)(12(1-t)n+1)}
  \leq \frac{1}{12n}-\frac{2}{6n+1}\leq 0
\end{equation}
for $n\geq 1$.  From this we obtain the interval expression
\begin{align}
H_{n}(\htt) &
    \in \frac{1}{2}\log(2\pi\htt(1-\htt)) - \frac{1}{2}\log(1+1/n) + 
    \left[0,\frac{1}{12n\htt(1-\htt)}\right],
    \label{e:Hn:asymp}
\end{align}
valid for $t\in[1/n,1-1/n]$. The boundary values of $H_{n}$ at $t=0$
and $t=1$ are $-\log(n+1)/2$.

The next auxiliary function is
\begin{equation}
  Y(t)=\frac{1}{e^{-t^2/2}}\int_t^\infty e^{-s^2/2}ds\in \left(\frac{t}{1+t^2},\frac{1}{t}\right)\text{\ \ for\ }t>0,\label{e:defY}
\end{equation}
where the bounds are from Ref.~\cite{mckay}. See this reference for a
summary of all properties of $Y$ mentioned here, or
Ref.~\cite{patel_jk:qc1996a} for more details. The function $Y$ is
related to the tail of the standard normal distribution, the
$Q$-function, by $Q(t)=e^{-t^{2}/2}Y(t)/\sqrt{2\pi}$.  The function $Y$
is monotonically decreasing, convex, $Y(0)=\sqrt{\pi/2}$, and it
satisfies the differential equation $\frac{d}{dt}Y(t) = tY(t)-1$.  We
make use of the following bounds involving $Y$:
\begin{equation}
-\log tY(t)
\in  
\left[0,\frac{1}{t^{2}}\right].
\label{e:BndY}
\end{equation}
The lower bound comes from the upper bound $1/t$ for $Y(t)$.  The
upper bound is from the lower bound $t/(1+t^{2})$ for $Y(t)$. 
Specifically, we compute $-\log(Y(t))\leq -\log(t/(1+t^{2})) = \log(t) +
\log(1+1/t^{2})\leq \log(t)+1/t^{2}$. 

With these definitions, we can express the $\log(p)$-values in terms
of their difference from $-\log(P_{\CH})$.

\begin{theorem}\label{t:pbr-ch}
For $0<\varphi\leq\htt<1$,
\begin{align}
-\log(P_{\PBR,n}(\htt|\varphi)) &= -\log(P_{\CH,n}(\htt|\varphi)) 
                 -\frac{1}{2}\log(n+1) + H_{n}(\htt)\label{e:t:pbr-ch-hn}\\
   &\in -\log(P_{\CH,n}(\htt|\varphi)) -\frac{1}{2}\log(n+1) + \frac{1}{2}\log(2\pi\htt(1-\htt)) - \frac{1}{2}\log(1+1/n)
  \notag\\
  &\hphantom{\in\;\;}+
  \left[0,\frac{1}{12n\htt(1-\htt)},\right] \label{e:t:pbr-ch}
\end{align}
\end{theorem}

\begin{proof}
The theorem is obtained by substituting definitions and then applying
the bounds of Eq.~\ref{e:Hn:asymp} on $H_{n}(t)$. Here are the details.
\begin{align}
-\log(P_{\PBR,n}(\htt|\varphi))
  &= -\log\left(\varphi^{n\htt}(1-\varphi)^{n(1-\htt)}(n+1)\binom{n}{n\htt}\right)
  \notag\\
  &= -\log\left(\left(\frac{\varphi}{\htt}\right)^{n\htt}\left(\frac{1-\varphi}{1-\htt}\right)^{n(1-\htt)}\right)\notag\\
  &\hphantom{=\;\;}
     -\log\left(\htt^{n\htt}(1-\htt)^{n(1-\htt))}(n+1)\binom{n}{n\htt}\right)\notag\\
  &=-\log(P_{\CH,n}(\htt|\varphi))-\frac{1}{2}\log(n+1)\notag\\
  &\hphantom{=\;\;}   -\log\left(\htt^{n\htt}(1-\htt)^{n(1-\htt))}\sqrt{n+1}\binom{n}{n\htt}\right)\notag\\
  &=-\log(P_{\CH,n}(\htt|\varphi))-\frac{1}{2}\log(n+1)+H_{n}(\htt).
\end{align}
It remains to substitute the interval expression for $H_{n}(\htt)$.
\end{proof}

\begin{theorem}\label{t:x-pbr}
  Define 
  \begin{equation}
     \lE_{n}(\htt|\varphi)
    =\min\left((\htt-\varphi)\sqrt{\frac{\pi n}{8\varphi(1-\varphi)}},1\right).
  \end{equation}
  Then for $0<\varphi<\htt<1$,
  \begin{align}
    -\log(P_{\X,n}(\htt|\varphi)) &\in
    -\log(P_{\PBR,n}(\htt|\varphi))+\log(n+1)-\log\left(\htt\sqrt{\frac{(1-\varphi)}{\varphi}}\right)
    \notag\\
    &\hphantom{{}={}\;\;}
    -\log\left(\sqrt{n}Y\left(\sqrt{\frac{n}{\varphi(1-\varphi)}}(\htt-\varphi)\right)\right)
    + \left[-\frac{\lE_{n}(\htt|\varphi)}{n(\htt-\varphi)},0\right],\label{e:t:x-pbr}\\
    -\log(P_{\X,n}(t|\varphi)) &\in -\log(P_{\CH,n}(t|\varphi))+\frac{1}{2}\log(n)
    -\log\left(\sqrt{\frac{\htt(1-\varphi)}{2\pi(1-\htt)\varphi}}\right)\notag\\
    &\hphantom{{}={}\;\;}
    -\log\left(\sqrt{n}Y\left(\sqrt{\frac{n}{\varphi(1-\varphi)}}(\htt-\varphi)\right)\right)
    + \left[-\frac{\lE_{n}(\htt|\varphi)}{n(\htt-\varphi)},\frac{1}{12n\htt(1-\htt)}\right].\label{e:t:x-ch}
  \end{align}
\end{theorem}

Observe that $\lE_{n}(t|\varphi)$ is $O(1)$ with respect to $n$ for
$t>\varphi$ constant.  The first term in the defining minimum is
smaller than $1$ only for $\varphi$ within less than one standard
deviation (which is $O(1/\sqrt{n})$) of $t$. It is defined so that the
primary dependence on the parameters is visible in the interval
bounds.

\begin{proof}
  For approximating $P_{\X}$, we apply Thm. 2 of Ref.~\cite{mckay}
  with the following sequence of substitutions, the first four of which
  expand the definitions in the reference:
  \begin{align}
    B(k;n,p)&\leftarrow\sum_{j=k}^{n}b(j;n,p),\notag\\
    b(k-1;n-1,p)&\leftarrow \binom{n-1}{k-1}p^{k-1}(1-p)^{n-k},\notag\\
    x&\leftarrow (k-pn)/\sigma,\notag\\
    \sigma&\leftarrow\sqrt{np(1-p)},\notag\\
    p&\leftarrow\varphi,\notag\\
    k&\leftarrow n\htt.
  \end{align}
  With the given substitutions and $Y(t)$ as defined by
  Eq.~\ref{e:defY}, we obtain for $\htt\geq\varphi$,
  \begin{align}
    -\log(P_{\X}) &\in -\log\left(\sqrt{n\varphi(1-\varphi)}\varphi^{n\htt-1}(1-\varphi)^{n(1-\htt)}\binom{n-1}{n\htt-1}\right) \notag\\
    &\hphantom{{}={}\;\;} -\log\left(Y\left(\frac{\sqrt{n}(\htt-\varphi)}{\sqrt{\varphi(1-\varphi)}}\right)\right) 
    + \left[-\frac{\lE_{n}(\htt|\varphi)}{n(\htt-\varphi)},0\right]
    \notag\\
    &=-\log\left(\varphi^{n\htt}(1-\varphi)^{n(1-\htt)}(n+1)\binom{n}{n\htt}\right)-\log\left(\frac{n\htt\sqrt{n\varphi(1-\varphi)}}{\varphi n(n+1)}\right) \notag\\
    &\hphantom{{}={}\;\;} -\log\left(Y\left(\frac{\sqrt{n}(\htt-\varphi)}{\sqrt{\varphi(1-\varphi)}}\right)\right) + \left[-\frac{\lE_{n}(\htt|\varphi)}{n(\htt-\varphi)},0\right]
    \notag\\
    &=-\log(P_{\PBR})+\log(n+1)-\log\left(\htt\sqrt{\frac{(1-\varphi)}{\varphi}}\right) \notag\\
    &\hphantom{{}={}\;\;} -\log\left(\sqrt{n}Y\left(\frac{\sqrt{n}(\htt-\varphi)}{\sqrt{\varphi(1-\varphi)}}\right)\right)
    + \left[-\frac{\lE_{n}(\htt|\varphi)}{n(\htt-\varphi)},0\right].
  \end{align}

  The second identity of the theorem follows by substituting the expression from
  Thm.~\ref{t:pbr-ch}.
\end{proof}

\MKc{Permanent comment to make the connection to Thm.~2~\cite{mckay}:

  The referenced theorem says: Let $0 < p < 1$, $n \ge 1$, and $pn \le k \le n$. Define
  $x = (k - pn)/\sigma$. Then 
  \begin{equation*}
    B(k; n, p) = \sigma b(k - 1; n - 1, p)Y (x) e^{
    E(k; n, p)/\sigma},
  \end{equation*}
  where $0 \le E(k; n, p) \le \min( \sqrt{\pi/8}, 1/x)$.
  We need the $1/x$ part of $E(k;n,p)$ here for asymptotic normality
  of the $\log(p)$-value differences.
}

We can eliminate the function $Y$ from the expressions by applying the
bounds from Eq.~\ref{e:BndY}.

\begin{corollary}\label{cor:x-ch}
  With the assumptions of Thm.~\ref{t:x-pbr},
  \begin{align}
    -\log(P_{\X,n}(t|\varphi)) &\in
    -\log(P_{\CH,n}(t|\varphi))+\frac{1}{2}\log(n)
    -\log\left(\frac{1-\varphi}{\htt-\varphi}\sqrt{\frac{\htt}{2\pi(1-\htt)}}\right)\notag\\
    &\hphantom{{}={}\;\;}
    +
    \left[-\frac{\lE_{n}(\htt|\varphi)}{n(\htt-\varphi)},\frac{\varphi(1-\varphi)}{(\htt-\varphi)^{2}n}+\frac{1}{12 n\htt(1-\htt)}\right].
  \end{align}
\end{corollary}

\begin{proof}
Define $c=(\htt-\varphi)/\sqrt{\varphi(1-\varphi)}$.
In view of Eq.~\ref{e:BndY}, we have
\begin{align}
-\log\left(\sqrt{n}Y\left(\sqrt{\frac{n}{\varphi(1-\varphi)}}(\htt-\varphi)\right)\right) &=
  -\log(\sqrt{n}Y(c\sqrt{n}))\notag\\
   &= \log(c)-\log(c\sqrt{n}Y(c\sqrt{n}))\notag\\
   &\in \log(c)+\left[0,\frac{1}{c^{2}n}\right].
\end{align}
Substituting in Eq.~\ref{e:t:x-ch} and simplifying the expression
gives the desired result.
\end{proof}

\subsection{Asymptotic Normality of the $\log(p)$-Values and Their
  Differences}

The main tool for establishing the asymptotic distribution of the
$\log(p)$-values is the ``delta method''. A version sufficient for our
purposes is Thm. 1.12 and Cor. 1.1 of Ref.~\cite{Shao}.  The notation
$X_{n}\xrightarrow{D}N(\mu,\sigma^{2})$ means that $X_{n}$ converges
in distribution to the normal distribution with mean $\mu$ and
variance $\sigma^{2}$.  By the central limit theorem,
$\Htheta_{n}=S_{n}/n$ satisfies
$\sqrt{n}(\Htheta_{n}-\theta)\xrightarrow{D}
N(0,\theta(1-\theta))$. An application of the delta method therefore
yields the next lemma.

\begin{lemma}\label{ANlemma}
  Let $F:\mathbb{R}\rightarrow\mathbb{R}$ be
  differentiable at $\theta$, with $F'(\theta)\neq
  0$. Then
  \begin{equation*}
    \sqrt{n}(F(\Htheta_{n})-F(\theta))\xrightarrow{D}
    N\left(0,F'(\theta)^{2}\theta(1-\theta)\right)
  \end{equation*}
\end{lemma}

\begin{theorem}\label{thm:gain_anorm}
  For $P=P_{\CH}$, $P=P_{\PBR}$ or $P=P_{\X}$, and $0<\varphi<\theta<1$ constant,
  the gain per trial $G_{n}(P)$ converges in distribution
  according to
  \begin{equation}
    \sqrt{n}(G_{n}(P)-\kl(\theta|\varphi))\xrightarrow{D}
    N(0,\sigma^{2}_{G}),
  \end{equation}
  with 
  \begin{equation*}
    \sigma_{G}^2= \theta(1-\theta)\left(\log\left(\frac{\theta}{1-\theta}\frac{1-\varphi}{\varphi}\right)\right)^2.
  \end{equation*}
\end{theorem}

\begin{proof} 
  Consider $P=P_{\CH}$ first. In Lem.~\ref{ANlemma},
  define $F(x)= \kl(x|\varphi) = x\log(x/\varphi) +
  (1-x)\log((1-x)/(1-\varphi))$ so that
  $F(\Htheta_{n})=G_{n}(P_{\CH})$. For the derivative of $F$ at $x=\theta$, we
  get
  \begin{equation}
    F'(\theta)=\log\left(\frac{\theta}{1-\theta}\frac{1-\varphi}{\varphi}\right).
  \end{equation}
  The theorem now follows for $P_{\CH}$ by applying Lem.~\ref{ANlemma}.

  Thm.~\ref{t:pbr-ch} and the law of large numbers imply that
  $(-\log(P_{\PBR})/\sqrt{n})-(-\log(P_{\CH})/\sqrt{n})$ converges in
  probability to $0$. Cor.~\ref{cor:x-ch} implies the same for
  $P_{\X}$, namely that
  $(-\log(P_{X})/\sqrt{n})-(-\log(P_{\CH})/\sqrt{n})$ converges in
  probability to $0$. In general, if $X_{n}-Y_{n}$ converges in
  probability to $0$ and $Y_{n}\xrightarrow{D} \mu$, then
  $X_{n}\xrightarrow{D}\mu$, see Ref.~\cite{brockwell:1991a},
  Prop. 6.3.3. The statement of the theorem to be proven now follows
  for $P=P_{\PBR}$ and $P=P_{\X}$ by comparison of
  $\sqrt{n}G_{n}(P_{\PBR})$ and $\sqrt{n}G_{n}(P_{\X})$ to
  $\sqrt{n}G_{n}(P_{\CH})$.
\end{proof}

The differences of the $\log(p)$-values have much tighter
distributions.  They are also asymptotically normal with scaling and
variances given in the next theorem. The differences are
$\Omega(\log(n))$ with standard deviations $O(1/\sqrt{n})$.

\begin{theorem}\label{thm:gap_anorm}
  Let $0<\varphi<\theta<1$ be constant. If $\theta\not=1/2$, then
  $P_{\PBR}/(\sqrt{n}P_{\CH})$ satisfies
  \begin{align}
    -\sqrt{n}\log\left(\frac{\sqrt{2\pi\theta(1-\theta)}P_{\PBR}}{\sqrt{n}P_{\CH}}\right) &\xrightarrow{D} 
    N\left(0,\frac{(1-2\theta)^{2}}{4\theta(1-\theta)}\right).
    \label{e:t:anpbr-ch}
  \end{align}
  If $\varphi\not=\theta(2\theta-1)$, then $\sqrt{n}P_{\X}/P_{\CH}$ satisfies
  \begin{align}
    -\sqrt{n}\log\left(\frac{\theta-\varphi}{1-\varphi}\sqrt{\frac{2\pi(1-\theta)}{\theta}}\frac{\sqrt{n}P_{\X}}{P_{\CH}}\right)
    &\xrightarrow{D}
    N\left(0,\frac{(\theta(1-2\theta)+\varphi)^{2}}{4(\theta-\varphi)^{2}\theta(1-\theta)}\right),\label{e:t:anx-ch}
  \end{align}  
\end{theorem}

\begin{proof}
  From Thm.~\ref{t:pbr-ch}, Eq.~\ref{e:t:pbr-ch} and the law of large
  numbers, we see that
  \begin{equation}
   \sqrt{n}\left( -\log\left(\frac{P_{\PBR}}{\sqrt{n}P_{\CH}}\right)-\log\left(\sqrt{2\pi\Htheta(1-\Htheta)}\right)\right)
  \end{equation}
  converges in probability to zero. From Lem.~\ref{ANlemma}
  and
  \begin{equation}
    \frac{d}{dx} \log(x(1-x))/2 = 
        \frac{1}{2x}-\frac{1}{2(1-x)} = \frac{1-2x}{2x(1-x)},
  \end{equation}
  we conclude
  \begin{equation}
   \sqrt{n}\left(\log\left(\sqrt{2\pi\Htheta(1-\Htheta)}\right)-\log\left(\sqrt{2\pi\theta(1-\theta)}\right)\right)\xrightarrow{D}
   N\left(0,\left(\frac{1-2\theta}{2\theta(1-\theta)}\right)^{2} \theta(1-\theta)\right).
  \end{equation}
  Combining the above observations gives Eq.~\ref{e:t:anpbr-ch}.

  Similarly, from Cor.~\ref{cor:x-ch} and taking note of the definition of $\lE_{n}(t|\varphi)$,
  \begin{equation}
    \sqrt{n}\left(
      -\log\left(\frac{\sqrt{n}P_{\X}}{P_{\CH}}\right) 
      -\log\left(\frac{\Htheta-\varphi}{1-\varphi}\sqrt{\frac{2\pi(1-\Htheta)}{\Htheta}}\right)
    \right)
  \end{equation}
  converges in probability to zero. The relevant derivative is 
  \begin{equation}
    \frac{d}{dx}\left(\log(x-\varphi)+\log((1-x)/x)/2\right)
     = \frac{1}{x-\varphi}-\frac{1}{2(1-x)}-\frac{1}{2x}
      = \frac{x(1-2x)+\varphi}{2(x-\varphi)x(1-x)},
  \end{equation}
  from which
  \begin{align}
   \sqrt{n}\left(\log\left(\frac{\Htheta-\varphi}{1-\varphi}\sqrt{\frac{2\pi(1-\Htheta)}{\Htheta}}\right)
   -\log\left(\frac{\theta-\varphi}{1-\varphi}\sqrt{\frac{2\pi(1-\htheta)}{\htheta}}\right)\right)
   \hspace{-2.5in}&\notag\\
   &\xrightarrow{D}
   N\left(0,\left(\frac{\theta(1-2\theta)+\varphi}{2(\theta-\varphi)\theta(1-\theta)}\right)^{2} \theta(1-\theta)\right),
  \end{align}
  and combining the two observations gives Eq.~\ref{e:t:anx-ch}.
\end{proof}

\subsection{Confidence Interval Endpoints}
\label{app:endpoints}

For the one-sided confidence intervals, we need to determine the lower
boundaries of acceptance regions, that is the confidence lower
bounds. By monotonicity of the $p$-values in $\varphi$, it suffices to
solve equations of the form $-\log(P(\htheta,\varphi))=\alpha$, where
$a=e^{-\alpha}$ is the desired significance level. Here we obtain
lower and upper bounds on the solutions $\varphi$.

To illuminate the asymptotic behavior of solutions $\varphi$ of
$-\log(P(\htheta,\varphi))=\alpha$, we reparametrize the
log-$p$-values so that our scale is set by an empirical
standard deviation, namely
$\hat\sigma=\sqrt{\htheta(1-\htheta)/n}$.  Thus we express the
solution as
\begin{equation}
  \varphi(\gamma,\htheta) = \htheta -\hat\sigma\gamma,
\end{equation}
in terms of a scaled deviation down from $\htheta$.
Inverting for $\gamma$ we get
\begin{equation}
  \gamma=\gamma(\varphi,\htheta)=\frac{\htheta-\varphi}{\hat\sigma}.
\end{equation}

\begin{theorem}\label{CHinterval}
  Let $0<\htheta<1$ and $\alpha>0$. Suppose that
  $\alpha\leq n\htheta^{2}(1-\htheta)^{2}/8$.  Then
  there is a solution $\gamma_\alpha>0$ of the identity
  $-\log(P_{\CH}(\htheta,\varphi(\gamma_\alpha,\htheta)))=\alpha$
  satisfying
  \begin{equation}
    \gamma_{\alpha} \in \sqrt{2\alpha}\left(1 + \frac{5}{2}\frac{\sqrt{\alpha}}{\sqrt{n\htheta(1-\htheta)}}[-1,1]\right)^{-1/2}.
  \end{equation}
\end{theorem}

The constants in this theorem and elsewhere are chosen for convenience,
not for optimality; better constants can be extracted from the proofs.
Note that the upper bound on $\alpha$ ensures that the reciprocal
square root is bounded away from zero. However, for the relative error to
go to zero as $n$ grows requires $\alpha=o(n)$.

\begin{proof}
  Consider the parametrized bound $\alpha\leq
  2n\htheta^{2}(1-\htheta)^{2}(1-a_{1})^{2}$, where later we set
  $a_{1}=3/4$ to match the theorem statement.  Let
  $F(\gamma)=-\log(P_{\CH}(\htheta,\varphi(\gamma,\htheta)))$.
  $F$ is continuous and monotone increasing.  A standard
  simplification of the Chernoff-Hoeffding bound noted in
  Ref.~\cite{Hoeffding} is
  \begin{equation}
    P_{\CH}\leq  e^{-2n(\htheta-\varphi)^{2}}=e^{-2\htheta(1-\htheta)\gamma^{2}}.\label{e:sch}
  \end{equation}  
  For $\varphi=\varphi(\gamma_\alpha,\htheta)$ solving the desired
  equation, we have $(\htheta-\varphi)\leq \sqrt{\alpha/(2n)}$ (by
  monotonicity), which in turn is bounded above according to
  $\sqrt{\alpha/2n}\leq
  \htheta(1-\htheta)(1-a_{1})\leq\htheta(1-a_{1})$, according to our
  assumed bound.  We conclude that $\varphi\geq a_{1}\htheta$.  For the
  solution $\gamma_\alpha$, we get $\gamma_{\alpha}\leq
  \sqrt{\alpha/(2\htheta(1-\htheta))}\leq
  \sqrt{n\htheta(1-\htheta)}(1-a_{1})$.
  
  We now Taylor expand $\kl(\htheta|\varphi)$ with remainder at
  third order around $\varphi=\htheta$. Write
  $f(x)=\kl(\htheta|\htheta-x)$, where we can restrict $x$ according
  to $\htheta\geq \htheta-x=\varphi\geq a_{1}\htheta$.  The derivatives of $f$ can
  be written explicitly as follows:
  \begin{equation}
    f^{(k)}(x) = 
    (k-1)!\frac{\htheta}{(\htheta-x)^{k}}
    - (-1)^{k-1}(k-1)!\frac{1-\htheta}{(1-\htheta+x)^{k}}.
  \end{equation}
  We have 
  \begin{align}
    f^{(1)}(0)&=0,\notag\\
    f^{(2)}(0)&= \frac{1}{\htheta}+\frac{1}{1-\htheta}=
    \frac{1}{\htheta(1-\htheta)},\notag\\
    f^{(3)}(x)&= 2\frac{\htheta}{(\htheta-x)^{3}} -2\frac{1-\htheta}{(1-\htheta+x)^{3}},\notag\\
    f^{(3)}(x) &\leq 2\frac{\htheta}{(\htheta-x)^{3}}
    \leq 2\frac{\htheta}{a_{1}^{3}\htheta^{3}}= 2\frac{1}{a_{1}^{3}\htheta^{2}},\notag\\
    f^{(3)}(x) &\geq -2\frac{1-\htheta}{(1-\htheta+x)^{3}}
    \geq  -2\frac{1-\htheta}{(1-\htheta)^{3}}
    = -2\frac{1}{(1-\htheta)^{2}},
  \end{align}
  since $0<a_{1}<1$. We use the bounds on $f^{(3)}(x)$ to bound the
  remainder in the Taylor expansion, where, to get cleaner
  expressions, we can decrease $\htheta$ and $1-\htheta$ to
  $\htheta(1-\htheta)$ in the denominators.
  \begin{equation}
    \kl(\htheta|\htheta-x)
      \in \frac{x^{2}}{2\htheta(1-\htheta)}
          + \frac{x^{3}}{3(\htheta(1-\htheta))^{2}}\left[-1,
             \frac{1}{a_{1}^{3}}\right].
  \end{equation}
  Substituting $x=\gamma_{\alpha}\sqrt{\htheta(1-\htheta)/n}$ gives
  \begin{equation}
    \alpha = -\log(P_{\CH}(\htheta,\varphi(\gamma_{\alpha},\htheta))) =
    n\kl(\htheta|\htheta-x)
    \in \frac{\gamma_{\alpha}^{2}}{2}\left(1 
    + \frac{2\gamma_{\alpha}}{3\sqrt{n\htheta(1-\htheta)}}\left[-1,
    \frac{1}{a_{1}^{3}}\right]\right).\label{e:t:chinterval:firstbnd}
  \end{equation}
  For $\htheta\leq 1/2$, $f^{(4)}(x)$ and $f^{(3)}(0)$ are
  non-negative, so we could have taken the lower bound in the interval
  to be zero for $\theta\leq 1/2$. For the theorem, we prefer not to
  separate the cases.

  We substitute the bound $\gamma\leq
  \sqrt{n\htheta(1-\htheta)}(1-a_{1})$ for the $\gamma$ multiplying
  the interval in Eq.~\ref{e:t:chinterval:firstbnd} and use the lower
  bound in the interval for the inequality
  \begin{equation}
     \alpha \geq \frac{\gamma^{2}}{2}\left(1-\frac{2(1-a_{1})}{3}\right).
  \end{equation}
  For the theorem, we have $a_{1}=3/4$, so $1-2(1-a_{1})/3=
  5/6$. Inverting the inequality for $\gamma$ gives
  $\gamma\leq 2\sqrt{3/5}\sqrt{\alpha}$.
  Now substituting this bound on $\gamma$ for the $\gamma$ multiplying
  the interval in Eq.~\ref{e:t:chinterval:firstbnd} gives
  \begin{equation}
    \alpha \in \frac{\gamma^{2}}{2}\left(1 +
       \frac{4}{\sqrt{15}}
        \frac{\sqrt{\alpha}}{\sqrt{n\htheta(1-\htheta)}}\left[-1,\frac{4^{3}}{3^{3}}\right]\right).
  \end{equation}
  By monotonicity of the appropriate operations, 
  \begin{equation}
    \gamma \in \sqrt{2\alpha}\left(1 + \frac{4}{\sqrt{15}}
      \frac{\sqrt{\alpha}}{\sqrt{n\htheta(1-\htheta)}}\left[-1,\frac{4^{3}}{3^{3}}\right]\right)^{-1/2}.
  \end{equation}
  For the theorem statement, we simplify the bounds with $1\leq
  4^{3}/3^{3}$ and $4^{4}/(3^{3}\sqrt{15})\leq 5/2$.  \MKc{Permanent
    comment to check the second inequality ``by hand'': Write it as $2^9 \leq
    3^3\times 5\sqrt{15}$, which expands to $512 \leq 135\sqrt{15}$. Write
    this as $512 \leq (128+7)\sqrt{16-1}$. Equivalently $1\leq
    (1+7/128)\sqrt{1-1/16}$.  We have $(1+7/128) \geq \sqrt{1+7/64}$
    (from the general inequality $(1+x/2)\geq \sqrt{1+x}$). Thus the
    right-hand side of the inequality is at least
    $\sqrt{(1+7/64)(1-4/64)}=\sqrt{1+3/64-28/(64^2)} \geq 1$,
    verifying the inequality.  }
\end{proof}

\begin{theorem}\label{thm:PBRinterval}
  Let $0<\htheta<1$ and $\alpha>0$. 
  Define $\Delta=\log(n+1)/2-H_{n}(\htheta)$. Suppose that
  $\alpha+\Delta\leq n\htheta^{2}(1-\htheta)^{2}/8$.
  Then there is a solution $\gamma_\alpha>0$ of the identity
  $-\log(P_{\PBR}(\htheta,\varphi(\gamma_\alpha,\htheta)))=\alpha$
  satisfying
  \begin{equation*}
    \gamma_\alpha
    \in 
    \sqrt{2(\alpha+\Delta)}\left(1+\frac{5}{2}\frac{\sqrt{\alpha+\Delta}}{\sqrt{n\htheta(1-\htheta)}}[-1,1]\right)^{-1/2}.
  \end{equation*}
\end{theorem}
\begin{proof}
  By Thm.~\ref{t:pbr-ch}, $-\log(P_{\CH})- (-\log(P_{\PBR})) =
  \Delta$.  If we define $\tilde{\alpha}=\alpha+\Delta$, then solving
  $-\log(P_{\PBR})=\alpha$ is equivalent to solving
  $-\log(P_{\CH})=\tilde{\alpha}$. Since $\Delta$ depends only on $n$
  and $\htheta$, $\tilde{\alpha}$ does not depend on $\gamma$. We
  can therefore apply Thm.~\ref{CHinterval} to get the desired
  bounds.
\end{proof}

\begin{theorem}\label{thm:exact_endpt}
  For $x\geq 0$, let $q(x) = -\log(e^{-x^{2}/2}Y(x)/\sqrt{2\pi}) =
  x^{2}/2+\log(2\pi)/2-\log(Y(x))$. Suppose that
  $0<\htheta<1$, and $\log(2)<\alpha\leq n\htheta^{2}(1-\htheta)^{2}/8$.
  Then there is a solution $\gamma_{\alpha}$ of the identity
  $-\log(P_{\X}(\htheta,\varphi(\gamma_{\alpha},\htheta))=\alpha$ satisfying
  \begin{align}
    \gamma_{\alpha} &\in \max\left(0,q^{-1}\left(\alpha\left(1+\frac{64\sqrt{\alpha}/(15\sqrt{15})}{\sqrt{n\htheta(1-\htheta)}}[-1,1]\right)+\frac{\sqrt{\pi/6}+8\sqrt{\alpha}/\sqrt{15}}{\sqrt{n\htheta(1-\htheta)}}[-1,1]\right)\right)\notag\\
    &\hphantom{\in\;\;}\times\left(1+\frac{2\sqrt{\alpha}/\sqrt{5}}{\sqrt{n\htheta(1-\htheta)}}[-1,1]\right),\label{e:t:exact_endpt}
  \end{align}
  where we extend $q^{-1}$ to negative values by $q^{-1}(y)=-\infty$ for
  $y\leq 0$ (if necessary) when evaluating this interval expression.
\end{theorem}

The function $q(x)$ is the negative logarithm of the
$Q$-function, which is the tail of the standard normal distribution.
The lower bound on $\alpha$ in Thm.~\ref{thm:exact_endpt} ensures that
there is a solution with $\gamma_{\alpha}>0$, because
$q(0)=\log(2)$. For reference, the constants multiplying the interval
expressions are $64/(15\sqrt{15})\approx 1.102$, $8/\sqrt{15}\approx
2.066$, $\sqrt{\pi/6}\approx 0.724$, $2/\sqrt{5}\approx 0.894$.
Note that in the large $n$ limit, where the $O(1/\sqrt{n})$ terms
  are negligible, the value of $\gamma_{\alpha}$ in
  Thm.~\ref{thm:exact_endpt} corresponds to the
  $(1-e^{-\alpha})$-quantile of the standard normal.

By monotonicity of $q^{-1}$, the explicit bounds in
Eq.~\ref{e:t:exact_endpt} are obtained by combining the lower or
the upper bounds in intervals in the expression. We remark that
$q^{-1}$ behaves well with respect to relative error for $\alpha$
large enough because of the inequalities
\begin{align}
  q^{-1}(y)/(1+q^{-1}(y)^{2})&\leq\frac{d}{dy} q^{-1}(y) \leq 1/q^{-1}(y), \notag\\
  q^{-1}(y)^{2}&\geq y-q(1)+1,&& \textrm{for $y\geq q(1)\approx1.841$},\notag\\
  q^{-1}(y)^{2}&\leq 2(y-\log(2)),&& \textrm{for $y\geq q(0)=\log(2)$},
  \label{e:q-1_bounds}
\end{align}
which we now establish. By implicit differentiation and from the
properties of $Y$ noted after Eq.~\ref{e:defY}, $\frac{d}{dy}
q^{-1}(y)|_{y=q(x)} = Y(x) \in [x/(1+x^{2}),1/x]$.  Therefore
$q^{-1}(y)/(1+q^{-1}(y)^{2})\leq\frac{d}{dy} q^{-1}(y)\leq
1/q^{-1}(y)$.  For $y\geq\log(2)$, we can integrate $\frac{d}{dz}
q^{-1}(z)^{2}=2q^{-1}(z)\frac{d}{dz} q^{-1}(z)\leq 2$ from $z=\log(2)$
to $y$ to show that
$q^{-1}(y)^{2}=q^{-1}(y)^{2}-q^{-1}(\log(2))^{2}\leq 2(y-\log(2))$,
making use of the identity $q^{-1}(\log(2))=0$.  Consider $y,z\geq
q(1)$. Since $q^{-1}(z)$ and $0\leq x\mapsto x^{2}/(1+x^{2})$ are
monotone increasing, $q^{-1}(z)^{2}/(1+q^{-1}(z)^{2}) \geq
q^{-1}(q(1))^{2}/(1+q^{-1}(q(1))^{2}) =1/2$, so the integral of
$\frac{d}{dz} q^{-1}(z)^{2}$ from $z=q(1)$ to $y$ with the lower bound
on $\frac{d}{dz} q^{-1}(z)$ gives $q^{-1}(y)^2-q^{-1}(q(1))^{2}=
q^{-1}(y)^{2}-1\geq y-q(1)$.

From the inequality $\frac{d}{dy} q^{-1}(y)\leq 1/q^{-1}(y)$ in
Eq.~\ref{e:q-1_bounds}, integration and monotonicity, for $0\leq z\leq
\delta$,
\begin{align}
  q^{-1}(\alpha-z)&\geq q^{-1}(\alpha)-\frac{z}{q^{-1}(\alpha-\delta)}
  \leq q^{-1}(\alpha)\left(1-\frac{z}{q^{-1}(\alpha-\delta)^{2}}\right),\notag\\
  q^{-1}(\alpha+z)&\leq q^{-1}(\alpha)+\frac{z}{q^{-1}(\alpha-\delta)}
  \geq q^{-1}(\alpha)\left(1+\frac{z}{q^{-1}(\alpha-\delta)^{2}}\right).
\end{align}
To determine the relative error, write $\delta'=\delta/\alpha$ to
obtain the interval inclusion
\begin{equation}
  q^{-1}(\alpha(1+\delta'[-1,1]))\subseteq q^{-1}(\alpha)\left(1+\frac{\alpha\delta'}{q^{-1}(\alpha(1-\delta'))^{2}}[-1,1]\right).
\end{equation}
For $\alpha(1-\delta')>q(1)$, the interval relationship can be
weakened to
\begin{equation}
q^{-1}(\alpha(1+\delta'[-1,1]))\subseteq q^{-1}(\alpha)\left(1+\frac{\alpha\delta'}{\alpha(1-\delta')-q(1)+1}[-1,1]\right).\label{eq:q-1relerr}
\end{equation}
The relative error on the right-hand side is given by the term
multiplying the interval, and can be written as
$\alpha\delta'/(\alpha-(\alpha\delta'+q(1)-1))$. If
$\alpha\delta'+q(1)-1\leq \alpha/2$, then the relative error is
bounded by $2\delta'$ which is twice the relative error of
$\alpha$. Of course, for the interval bounds to converge, we need
$\alpha=o(n)$.

\begin{proof}
  As in the proof of Thm.~\ref{CHinterval}, consider the parametrized
  bound $\alpha\leq 2n\htheta^{2}(1-\htheta)^{2}(1-a_{1})^{2}$, where
  later we set $a_{1}=3/4$ to match the statement of
  Thm.~\ref{thm:exact_endpt}. From the
  Chernoff-Hoeffding bound, we get $\varphi\geq a_{1}\htheta$ and
  $\gamma_\alpha\leq \sqrt{\alpha/(2\htheta(1-\htheta))}\leq
  \sqrt{n\htheta(1-\htheta)}(1-a_{1})$.

  Define $\tilde\gamma=(\htheta-\varphi)/\sqrt{\varphi(1-\varphi)/n}$.
  We start from Eq.~\ref{e:t:x-ch}, rewritten as follows:
  \begin{align}
    -\log(P_{\X}) &\in n\kl(\htheta|\varphi)+
    \frac{1}{2}\log(2\pi)-\log Y\left(\tilde\gamma\right)
    -\frac{1}{2}\log\left({\frac{\htheta(1-\varphi)}{(1-\htheta)\varphi}}\right)
    \notag\\
    &\hphantom{{}={}\;\;} +
    \left[-\frac{\lE_{n}(\htheta|\varphi)}{n(\htheta-\varphi)},\frac{1}{12n\htheta(1-\htheta)}\right].\label{e:t:pxinkl}
  \end{align}
  If $\tilde\gamma\geq \sqrt{8/\pi}\approx 1.6$,
  $\lE_{n}(\htheta|\varphi)=1$.  For better bounds at small values
    of $\tilde\gamma$, we use the other alternative in the definition
  of $\lE_{n}$, according to which the lower bound in the last
  interval of Eq.~\ref{e:t:pxinkl} is
  \begin{equation}
    -\frac{\lE_{n}(\htheta|\varphi)}{n(\htheta-\varphi)}
    \geq -\frac{\sqrt{\pi/8}}{\sqrt{n\varphi(1-\varphi)}}
    \geq -\frac{\sqrt{\pi/8}}{\sqrt{n a_{1}\htheta(1-\varphi)}}
    \geq -\frac{\sqrt{\pi/8}}{\sqrt{n a_{1}\htheta(1-\htheta)}}.
  \end{equation}
  
  Next we approximate
  $n\kl(\htheta|\varphi)$ in terms of $\tilde\gamma$ instead of
  $\gamma$. We still write the interval bounds in terms of $\gamma$.
  Let $f(x)=\kl(\varphi+x|\varphi)$.  We are concerned with the range
  $0\leq x \leq \htheta-\varphi$, with $\varphi\geq a_{1}\htheta$. We
  have
  \begin{align}
    f^{(1)}(x) &= \log((\varphi+x)/\varphi)-\log((1-\varphi-x)/(1-\varphi)) \notag\\
    f^{(2)}(x) &= \frac{1}{\varphi+x} + \frac{1}{1-\varphi-x}\notag\\
    &= \frac{1}{(\varphi+x)(1-\varphi-x)}\notag\\
    f^{(3)}(x) &= -\frac{1}{(\varphi+x)^{2}} + \frac{1}{(1-\varphi-x)^{2}}\notag\\
    &= -\frac{1-2(\varphi+x)}{(\varphi+x)^{2}(1-\varphi-x)^{2}}\notag\\
    |f^{(3)}(x)| &\leq  \frac{1}{a_{1}^{2}\htheta^{2}(1-\htheta)^{2}},
  \end{align}
  yielding
  \begin{equation}
    \kl(\varphi+x|\varphi) \in \frac{x^{2}}{2\varphi(1-\varphi)}
    + \frac{x^{3}}{6a_{1}^{2}\htheta^{2}(1-\htheta)^{2}}[-1,1] ,
  \end{equation}
  and with $x=\tilde\gamma\sqrt{\varphi(1-\varphi)/n}=\gamma\sqrt{\htheta(1-\htheta)/n}$,
  \begin{equation}
    n\kl(\htheta|\varphi) \in \frac{\tilde\gamma^{2}}{2}
    + \frac{\gamma^{3}}{6a_{1}^{2}\sqrt{n\htheta(1-\htheta)}}[-1,1].
  \end{equation}
  For the fourth term on the right-hand side of Eq.~\ref{e:t:pxinkl},
  \begin{equation}
    \frac{d}{dx}\log\left(\frac{\htheta(1-\htheta+x)}{(1-\htheta)(\htheta-x)}\right)
    =\frac{1}{1-\htheta+x}+\frac{1}{\htheta-x}=\frac{1}{(1-\htheta+x)(\htheta-x)},
  \end{equation}
  whose
  absolute value is bounded by $1/(a_{1}\htheta(1-\htheta))$ for $x$ in the 
  given range.  Thus
  \begin{equation}
    \log\left(\frac{\htheta(1-\varphi)}{(1-\htheta)\varphi}\right)
    \in \frac{\gamma}{a_{1}\sqrt{n\htheta(1-\htheta)}}[-1,1].
  \end{equation}

  Since $P_{\X}\leq P_{\CH}$, we can also use
  the bound $\gamma\leq 2\sqrt{3/5}\sqrt{\alpha}$ obtained in the proof of
  Thm.~\ref{CHinterval}. Substituting $a_{1}=3/4$ as needed, the
  equation to solve is now
  \begin{align}
    \alpha &\in \frac{\tilde\gamma^{2}}{2}+
    \frac{1}{2}\log(2\pi)-\log Y\left(\tilde\gamma\right)
    \notag\\
    &\hphantom{{}={}\;\;} +
    \frac{8}{\sqrt{15}}\frac{\sqrt{\alpha}}{\sqrt{n\htheta(1-\htheta)}}[-1,1]
    + 
    \frac{64}{15\sqrt{15}}\frac{\sqrt{\alpha}^{3}}{\sqrt{n\htheta(1-\htheta)}}[-1,1] \notag\\
    &\hphantom{{}={}\;\;}+
    \left[-\frac{\sqrt{\pi/6}}{\sqrt{n\htheta(1-\htheta)}},\frac{1}{12n\htheta(1-\htheta)}\right].
  \end{align}
  The sum of the first three terms evaluates to $q(\tilde\gamma)$.
  The remaining terms are now independent of $\gamma$ and are of order
  $1/\sqrt{n}$. They can be merged by means of common bounds using
  $2n\htheta(1-\htheta)\geq\sqrt{n\htheta(1-\htheta)}$, since
  $n\htheta(1-\htheta)\geq 1/2$ for our standing assumptions that
  $n\geq 1$ and $\htheta n$ is an integer different from $0$ and $n$.
  Consequently, $12n\htheta(1-\htheta)\geq 6\sqrt{n\htheta(1-\htheta)}
  \geq\sqrt{6/\pi}\sqrt{n\htheta(1-\htheta)}$.  The interval bounds
  then combine conservatively to
  \begin{equation}
    \frac{\sqrt{\pi/6}+8\sqrt{\alpha}/\sqrt{15}+64\sqrt{\alpha}^{3}/(15\sqrt{15})}{\sqrt{n\htheta(1-\htheta)}}.
  \end{equation}
  We can now write
  \begin{equation}
    \alpha \in q(\tilde\gamma)
    +
    \frac{\sqrt{\pi/6}+8\sqrt{\alpha}/\sqrt{15}+64\sqrt{\alpha}^{3}/(15\sqrt{15})}{\sqrt{n\htheta(1-\htheta)}}[-1,1],
  \end{equation}
  which holds iff
  \begin{equation}
   q(\tilde\gamma) \in \alpha\left(1+\frac{64\sqrt{\alpha}/(15\sqrt{15})}{\sqrt{n\htheta(1-\htheta)}}[-1,1]\right)+\frac{\sqrt{\pi/6}+8\sqrt{\alpha}/\sqrt{15}}{\sqrt{n\htheta(1-\htheta)}}[-1,1].
  \end{equation}
  By monotonicity of $q$ and extending $q^{-1}$ to negative arguments
  as mentioned in the statement of Thm.~\ref{thm:exact_endpt} if
  necessary, the constraint is equivalent to
  \begin{equation}
    \tilde\gamma\in q^{-1}\left(\alpha\left(1+\frac{64\sqrt{\alpha}/(15\sqrt{15})}{\sqrt{n\htheta(1-\htheta)}}[-1,1]\right)+\frac{\sqrt{\pi/6}+8\sqrt{\alpha}/\sqrt{15}}{\sqrt{n\htheta(1-\htheta)}}[-1,1]\right).\label{e:t:tildegamma}
  \end{equation}
  For $\alpha>\log(2)$, we know that $\tilde\gamma>0$, so we can add 
  $\max(0,\ldots)$ as in the theorem statement. 

  To determine the interval equation for $\gamma$, we have
  $\gamma=\tilde\gamma\sqrt{\varphi(1-\varphi)/(\htheta(1-\htheta))}$.
  We use the first-order remainder to bound the factor on the right-hand side.
  For this consider the numerator, and write $g(x)=
  \sqrt{(\htheta-x)(1-\htheta+x)}$ with $0\leq x\leq \htheta-\varphi$.
  We have
  \begin{align}
    g^{(1)}(x)
    &= \frac{2(\htheta-x)-1}{2\sqrt{(\htheta-x)(1-\htheta+x)}},\\
    |g^{(1)}(x)| &\leq\frac{1}{2\sqrt{a_{1}\htheta(1-\htheta)}}\notag\\
    &= \frac{1}{\sqrt{3\htheta(1-\htheta)}},\\
    g(x)&\in
    \sqrt{\htheta(1-\htheta)}+\frac{x}{\sqrt{3\htheta(1-\htheta)}}[-1,1].
  \end{align}
  With $x=\gamma\sqrt{\htheta(1-\htheta)/n}$
  and the bound of $\gamma\leq2\sqrt{3/5}\sqrt{\alpha}$, we get
  \begin{equation}
    \gamma\in\tilde\gamma\left(1+
       \frac{2\sqrt{\alpha}/\sqrt{5}}{\sqrt{n\htheta(1-\htheta)}}[-1,1]\right).
  \end{equation}
  The theorem follows by composing
  this constraint with Eq.~\ref{e:t:tildegamma}.
\end{proof}


\begin{acknowledgments}
  This work includes contributions of the National
  Institute of Standards and Technology, which are not subject to
  U.S. copyright. Y. Z. would like to acknowledge supports through 
  the Ontario Research Fund (ORF), the Natural Sciences and Engineering 
  Research Council of Canada (NSERC), and Industry Canada. 
\end{acknowledgments}

\bibliographystyle{plain}
\bibliography{paper}

\end{document}